\documentclass[10pt]{amsart}
\usepackage{amssymb,amsmath,txfonts,amsthm}
\usepackage{mathrsfs}
\usepackage{hyperref}
\usepackage{mathrsfs}
\newcommand\blfootnote[1]{%
  \begingroup
  \renewcommand\thefootnote{}\footnote{#1}%
  \addtocounter{footnote}{-1}%
  \endgroup
}
\newtheorem{theorem}{Theorem}
\newtheorem{prop}{Proposition}
\newtheorem{lemma}{Lemma}
\newtheorem{remark}{Remark}
\newtheorem{claim}{Claim}
\newtheorem{definition}{Definition}
\newtheorem{cor}{Corollary}
\newtheorem{conj}{Conjecture}

\numberwithin{equation}{section}
\numberwithin{theorem}{section}
\numberwithin{definition}{section}
\numberwithin{cor}{section}
\numberwithin{prop}{section}
\numberwithin{remark}{section}
\numberwithin{claim}{section}
\numberwithin{lemma}{section}

\def\Xint#1{\mathchoice
  {\XXint\displaystyle\textstyle{#1}}%
  {\XXint\textstyle\scriptstyle{#1}}%
  {\XXint\scriptstyle\scriptscriptstyle{#1}}%
  {\XXint\scriptscriptstyle\scriptscriptstyle{#1}}%
  \!\int}
\def\XXint#1#2#3{{\setbox0=\hbox{$#1{#2#3}{\int}$}
  \vcenter{\hbox{$#2#3$}}\kern-.5\wd0}}

\def\dashint{\Xint-}

\author{Gang Liu}
\address{Department of Mathematics\\Northwestern University\\Evanston, IL 60208}
\email{gang.liu@northwestern.edu}
\title[Uniformization]{On Yau's uniformization conjecture}
\date{}
\begin{document}
\begin{abstract}
Let $M^n$ be a complete noncompact K\"ahler manifold with nonnegative bisectional curvature and maximal volume growth, we prove that $M$ is biholomorphic to $\mathbb{C}^n$. 
This confirms the uniformization conjecture of Yau when $M$ has maximal volume growth.
\end{abstract}
\maketitle
\section{\bf{Introduction}}
\blfootnote{The author was partially supported by NSF grant DMS 1406593.}
In \cite{[Y1]}, Yau proposed the uniformization conjecture:
\begin{conj}\label{conj1}
Let $M^n$ be a complete noncompact K\"ahler manifold with positive bisectional curvature. Then $M$ is biholomorphic to $\mathbb{C}^n$.
\end{conj}
Conjecture \ref{conj1} is open so far. However, there has been much important progress. In earlier works, Mok-Siu-Yau \cite{[MSY]} and Mok \cite{[Mo]} considered embedding by using holomorphic functions of polynomial growth. Later, with the K\"ahler-Ricci flow, results were improved significantly. For example, in \cite{[CT]}, A. Chau and L. F. Tam proved that a complete noncompact K\"ahler manifold with bounded nonnegative bisectional curvature and maximal volume growth is biholomorphic to $\mathbb{C}^n$.
See also \cite{[Shi1]}\cite{[Shi2]}\cite{[CZ]}\cite{[N2]}\cite{[CT]}\cite{[CT1]}\cite{[CT2]}\cite{[CTZ]}\cite{[HT]}\cite{[He]}\cite{[LT]} for related works.

In \cite{[L1]}-\cite{[L5]}, we introduced a new method to study the conjecture. The basic strategy is to consider the Gromov-Hausdorff limit of K\"ahler manifolds with bisectional curvature lower bound. For instance, in \cite{[L3]}, it was proved that if a complete noncompact K\"ahler manifold has nonnegative bisectional curvature and maximal volume growth, then it is biholomorphic to an affine algebraic variety. In this paper, we shall continue to study the conjecture by this strategy. The main theorem is the following:

\begin{theorem}\label{thm1}
Let $M^n$ be a complete noncompact K\"ahler manifold with nonnegative bisectional curvature. Assume $M$ has maximal volume growth, then we can find $n$ polynomial growth holomorphic functions $f_1, .., f_n$ which give a biholomorphism from $M$ to $\mathbb{C}^n$.
\end{theorem}

We shall also study the case when the bisectional curvature has a lower bound.
\begin{definition}\cite{[LW]} \cite{[TY]}
On a K\"ahler manifold $M^n$, we say the bisectional curvature is greater than or equal to $K$ ($BK\geq K$), if
\begin{equation}
\frac{R(X, \overline{X}, Y, \overline{Y})}{||X||^2||Y||^2+|\langle X, \overline{Y}\rangle|^2}\geq K 
\end{equation}
for any two nonzero vectors $X, Y\in T^{1, 0}M$.\end{definition} Observe that the equality holds for complex space forms.
The bisectional curvature lower bound condition is weaker than the sectional curvature lower bound, while stronger than the Ricci curvature lower bound.

\begin{theorem}\label{thm2}
Let $(M^n_i, p_i)$ be a sequence of complete (compact or noncompact) K\"ahler manifolds with bisectional curvature lower bound $-1$. Assume $vol(B(p_i, 1))\geq v>0$. Suppose $(X, p)$ is the pointed Gromov-Hausdorff limit of $(M_i, p_i)$. Then 
\begin{itemize}
\item $(X, p)$ is homeomorphic to a normal complex analytic space with singularity of complex codimension at least $4$.  In particular, if $n\leq 3$, $X$ is a complex manifold. 
\item $X$ is homeomorphic to a manifold. 
If the diameters are uniformly bounded, $X$ is homeomorphic to $M_i$ for all large $i$. 
\end{itemize}
\end{theorem}
\begin{remark}

In the topological sense, the second part is the complex analogue of Perelman's stability theorem \cite{[P]}\cite{[Ka]} for Riemannian manifolds with sectional curvature lower bound. If the bisectional curvature lower bound is replaced by two side bounds of Ricci curvature, then $X$ could have singularity of complex codimension $2$. See for example, \cite{[DS]}\cite{[Ti5]}.
\end{remark}
\begin{remark}
Our original approach to theorem \ref{thm1} is to prove sufficient regularity for the tangent cone. For instance, by using the the regularity result of theorem \ref{thm2}, 
 we can prove theorem \ref{thm1} for $n\leq 3$. It is interesting to note that the regularity in theorem \ref{thm2} is very closely related with a conjecture of Shokurov (\cite{[Sho]}, conjecture 2) in algebraic geometry.
 For instance, if Shokurov conjecture is true, then the limit space $X$ is complex analytically smooth. So far, the Shokurov conjecture is only solved for dimension less than or equal to three (this is responsible for the dimension restriction $n\leq 3$ for theorem \ref{thm1}).
For some details on the Shokurov conjecture, one can refer to \cite{[Mc]}. In the current version, we bypass the difficulty in algebraic geometry by using a different method.
\end{remark}

We follow the strategies in \cite{[L3]}\cite{[L4]}\cite{[DS]}\cite{[DS2]}.
The new thing is to solve $\overline\partial$ equation on the holomorphic tangent bundle. Eventually we construct a global integrable holomorphic vector field retracting to a point. This gives us the desired biholomorphism from $M$ to $\mathbb{C}^n$. For the proof of theorem \ref{thm2}, we need some algebro-geometric results by D. Mumford \cite{[Mu]} and M. Mclean \cite{[Mc]}

This paper is organized as follows.
Section $2$ is some basic preliminary results. In section $3$, we solve the $\overline\partial$ equation on the holomorphic tangent bundle. This requires the full power of nonnegativity of bisectional curvature.
Section $4$ is the immediate application to topology of complete K\"ahler manifolds with nonnegative bisectional curvature and maximal volume growth. The proof of theorem \ref{thm2} is presented in section $5$.
In the last section, we prove the main theorem \ref{thm1}.

\begin{center}
\bf  {\quad Acknowledgment}
\end{center}
The author would like to express his deep gratitude to Professors J. Lott and J. Wang for constant encouragements and discussions. He thanks Professors A. Chau, A. Naber, L. F. Tam and S. T. Yau for the interest. 
Finally, he thanks Professors J. P. Demailly, Y. Eliashberg, M. Mclean and Y. T. Siu for answering questions which are closely related to this paper.

\section{\bf{Preliminaries}}
For the basic theory of Gromov-Hausdorff convergence, we refer to \cite{[G]}. See also preliminaries of \cite{[L1]}-\cite{[L5]}.

H\"ormander $L^2$ theory:
The following result can be found on page $37$-$38$ of \cite{[D]}. 
\begin{prop}\label{prop1}
Let $(X, \omega)$ be a connected K\"ahler manifold which is not necessarily complete. Assume $X$ is Stein. Let $(F, h)$ be a Hermitian holomorphic vector bundle over $X$ ($h$ is the metric). Assume the curvature operator $A = [\sqrt{-1}\Theta_{F, h}, \Lambda]$ is positive definite everywhere on $\Lambda^{n, 1}T^*_X\otimes F$. Then for any form $g\in L^2(M, \Lambda^{n, 1}T^*_X\otimes F)$ satisfying $\overline\partial g = 0$ and $\int_X\langle A^{-1}g, g\rangle\omega^n <+\infty$, there exists $f\in L^2(X, \Lambda^{n, 0}T^*_X\otimes F)$ such that $\overline\partial f = g$ and $\int_X|f|^2\omega^n\leq\int_X\langle A^{-1}g, g\rangle\omega^n$.
\end{prop}

Let $(F, h)$ be a Hermitian holomorphic vector bundle.  Let $z_1, ..., z_n$ be a local holomorphic chart and $e_1, .., e_m$ be an orthonormal frame of $F$. Let \begin{equation}\label{eq1}\sqrt{-1}\Theta_{F, h}=\sum\limits_{j, k, \lambda, \mu}c_{jk\lambda\mu}dz_j\wedge d\overline{z_k}\otimes e^*_\lambda\otimes e_\mu.
\end{equation}By using the metric $h$, we identify the curvature tensor $\Theta$ with a Hermitian form \begin{equation}\label{eq2}\tilde\Theta_{F, h}(\xi, v) = c_{jk\lambda\mu}\xi_j\overline{\xi}_kv_\lambda\overline{v}_\mu\end{equation} on $T^{1, 0}X\otimes F$. The next definition appears on page 27-28 of \cite{[D]}. 
\begin{definition}\label{def1} We say $(F, h)$ is

(a) Nakano positive if $\tilde\Theta_{F, h}(\tau)>0$ for all nonzero tensors $\tau = \sum\tau_{j\lambda}\frac{\partial}{\partial z_j}\otimes e_\lambda\in T_X\otimes F$.

(b) Griffiths positive if $\tilde\Theta_{F, h}(\xi\otimes v)>0$ for all nonzero decomposable tensors $\xi\otimes v\in T_X\otimes F$.
\end{definition}
The computation on page $35$ of \cite{[D]} shows that if $F$ is Nakano positive, then $[\sqrt{-1}\Theta_{F, h}, \Lambda]$ is positive on $(n, 1)$ forms with values in $F$. More explicitly, by equation $(4.8)$ on page $35$ of \cite{[D]}, we have that if $\tilde\Theta_{F, h}(\tau)\geq c|\tau|^2$, then \begin{equation}\label{eq3}\langle[\sqrt{-1}\Theta_{F, h}, \Lambda]u, u\rangle\geq c|u|^2\end{equation} for any $u\in \Lambda^{n, 1}T^*_M\otimes F$.
We also need the following
\begin{prop}\label{prop2}
For any Hermitian holomorphic vector bundle $E$, if $E$ is Griffiths positive, then $E\otimes \det(E)$ is Nakano positive. In fact, 
if $\tilde\Theta_{E, h}(\xi\otimes v)\geq c|\xi\otimes v|^2\geq 0$, then $\tilde\Theta_{E\otimes\det E, h}(\tau)\geq c|\tau|^2\geq 0$.\end{prop}

The proof can be found on page $93$ of \cite{[D]}. Notice the proof there also gives the second statement. More precisely, check the last equation of the proof which appears on page $95$.

Next we introduce a gluing technique:
\begin{definition}\label{def2}
Let $\chi$ be a strictly increasing continuous function over $\mathbb{R}^+$ and $\chi(0) = 0$. A metric space $X$ is $\chi$-connected if for any two points $x_1, x_2\in X$, we can find a curve $\gamma$ connecting $x_1, x_2$ so that the diameter of $\gamma$ is bounded by 
$\chi(d(x_1, x_2))$.
\end{definition}
We will need the gluing theorem which appears in \cite{[Ka]}\cite{[P]}:
\begin{prop}\label{prop3}[Gluing theorem]
 Let $X$ be a compact topological manifold which is also a metric space. Let $U_\alpha$ be a finite open covering of $X$. Given a function $\chi_0$, there exists $\delta=\delta(X, \chi_0, \{U_\alpha\}, )$ so that the following holds: Given a $\chi_0$-connected topological manifold $\tilde{X}$ (metric space), an open cover of $\tilde{X}$ $\{\tilde{U}_\alpha\}$, a $\delta$-Hausdorff approximation $\varphi: X\to\tilde{X}$ and a family of homeomorphisms $\varphi_\alpha$: $U_\alpha\to\tilde{U}_\alpha$, $\delta$-close to $\varphi$, then there exists a homeomorphism $\overline\varphi: X\to\tilde{X}$, $\chi(\delta)$-close to $\varphi$.
 \end{prop}
 
 In this paper, we will denote by $\Phi(u_1,..., u_k|....)$ any nonnegative functions depending on $u_1,..., u_k$ and some additional parameters such that when these parameters are fixed, $$\lim\limits_{u_k\to 0}\cdot\cdot\cdot\lim\limits_{u_1\to 0}\Phi(u_1,..., u_k|...) = 0.$$  Let $C(n), C(n, v)$ be large positive constants depending only on $n$ or $n, v$; $c(n), c(n, v)$ be small positive constants depending only on $n$ or $n, v$. The values might change from line to line. Let $\dashint$ be the average of an integral.
  
\section{\bf{Construction of retracting holomorphic vector fields}}
In this section, we shall construct retracting holomorphic vector fields on geodesic balls which are Gromov-Hausdorff close to metric cones. The argument is crucial for all the results in this paper.
\begin{prop}\label{prop4}
Give any $n\in\mathbb{N}$ and $v>0$, there exists $\epsilon=\epsilon(n, v)>0$ so that the following holds:
Let $(M^n, p)$ be a complete K\"ahler manifold with nonnegative bisectional curvature. If $vol(B(p, \frac{1}{\epsilon}))\geq v\frac{1}{\epsilon^{2n}}$ and $d_{GH}(B(p, \frac{1}{\epsilon}), B_W(o, \frac{1}{\epsilon}))<\epsilon$ for some metric cone $(W, o)$, then there exists a holomorphic vector field $Z$ on some open set $U\supset B(p, 1)$ so that 
the flow $\sigma_t$ generated by $-Z$ retracts to a point $\tilde{p}$ where $d(p, \tilde{p})=\Phi(\epsilon|n, v)$. Furthermore, $\sigma_t(B(p, \frac{1}{2}))\subset B(p, 1)$ for all $t\geq 0$.\end{prop}

\begin{proof}
Assume the proposition is not true. Then there exist some $n, v$ so that for any $i\in\mathbb{N}$, we can find a complete K\"ahler manifold $(M^n_i, p_i)$ which does not satisfy the proposition. Furthermore, $M_i$ has nonnegative bisectional curvature and for metric cones $(V_i, o'_i)$,
\begin{equation}\label{eq4}
vol(B(p_i, i))\geq vi^{2n}, d_{GH}(B(p_i, i), B_{V_i}(o'_i, i))<\frac{1}{i}.\end{equation} 
 By passing to a subsequence if necessary, we may assume $(M_i, p_i, d_i)$ pointed converges to a metric cone $(M_\infty, p_\infty, d_\infty)$. Let $R$ be a large number depending only on $n, v$. The value will be determined later. Let $r_i$ be the distance function to $p_i$. According to Cheeger-Colding theory \cite{[CC1]} and claim $5.1$ in \cite{[L3]}, for sufficiently large $i$, we can find smooth functions $u_i$ such that \begin{equation}\label{eq5}\int_{B(p_i, 4R)}|\nabla u_i-\nabla \frac{1}{2}r_i^2|^2 + |\nabla^2u_i-g_i|^2<\Phi(\frac{1}{i}|R)\end{equation} \begin{equation}\label{eq6}|u_i-\frac{r_i^2}{2}|<\Phi(\frac{1}{i}|R)\end{equation} \begin{equation}\label{eq7}|\nabla u_i|\leq C(n)r_i\end{equation} on $B(p_i, 4R)$.
Now define a $(1, 0)$ type vector field $\tilde{Z_i} = \nabla u_i - \sqrt{-1} J\nabla u_i$. Then we have that \begin{equation}\label{eq8}\int_{B(p_i, 4R)}|\overline\partial \tilde{Z}_i|^2< \Phi(\frac{1}{i}|R).\end{equation} The idea is to perturb $\tilde{Z_i}$ so that it becomes holomorphic. 
According to proposition $5.1$ of \cite{[L3]}, we can find a smooth function $v_i$ on $B(p_i, \frac{R}{2})$ such that \begin{equation}\label{eq9}0\leq v_i\leq C(R, n), \sqrt{-1}\partial\overline\partial v_i\geq c(n, v)\omega_i>0.\end{equation}\begin{equation}\label{eq10}\min\limits_{y\in\partial B(p_i, \frac{R}{20})}v_{i}(y)> 4\sup\limits_{y\in B(p_i, \epsilon_0\frac{R}{20})}v_{i}(y).\end{equation} for some $\epsilon_0(n, v)>0$. Now we fix $R=R(n, v)$ so that $R\epsilon_0>2000$. Let $\Omega_i$ be the connected component of $\{y\in B(p_i, \frac{R}{20})|v_{i}(y)< 2\sup\limits_{B(p_i, \frac{\epsilon_0R}{20})}v_{i}\}$ containing $B(p_i, \frac{\epsilon_0R}{20})$. Then $\Omega_i$ is a Stein manifold containing $B(p_i, 100)$.
Consider the metric \begin{equation}\label{eq11}g_i'  = e^{-v_i}g_i\end{equation} on the tangent bundle $T^{1, 0}M_i$. The curvature of $g_i'$ satisfies $\Theta_{g_i'} = \Theta_{g_i}+\sqrt{-1}\partial\overline\partial v_i\otimes Id$. Then we find that \begin{equation}\label{eq12}\tilde\Theta(\xi\otimes u)=\langle\Theta_{g_i'}(\xi\otimes u), \overline{\xi\otimes u}\rangle \geq \frac{1}{2}c(n, v)|\xi\otimes u|^2.\end{equation}  We have used that the bisectional curvature is nonnegative. This implies that $(T^{1, 0}M, g_i')$ is Griffiths positive.

$g_i'$ induces a metric on the anti-canonical line bundle $K^{-1}M_i$.
Take $F = T^{1, 0}_{M_i}\otimes K^{-1}(M_i)$. Let the metric $h$ on $F$ be induced by $g_i'$ on both the tangent bundle and $K^{-1}M_i$. Let the metric $\tilde{h}$ on $F$ be induced by the K\"ahler metric.
According to (\ref{eq9}) and (\ref{eq11}), \begin{equation}\label{eq13}c(n, v)\tilde{h}\leq h\leq C(n, v)\tilde{h}\end{equation} on $B(p_i, \frac{R}{2})$.  By proposition \ref{prop2}, $F$ is Nakano positive.
Therefore (\ref{eq3}) holds for $(F, h)$.  Write $T^{1, 0}(M_i) = \Lambda^{n, 0}T^*M_i\otimes F$. By applying proposition \ref{prop1} to Stein manifold $\Omega_i$ and $(F, h)$, we obtain a $(1, 0)$ type vector field $Y_i$ satisfying \begin{equation}\label{eq14}\overline\partial Y_i = \overline\partial \tilde{Z}_i,\end{equation}
\begin{equation}\label{eq15}\int_{\Omega_i}|Y_i|^2\leq \frac{1}{c(n, v)}\int_{\Omega_i}|\overline\partial \tilde{Z_i}|^2<\Phi(\frac{1}{i}).\end{equation}
In (\ref{eq15}), the norms are induced by the K\"ahler metric of $M_i$ (we have used (\ref{eq13})).
Therefore, \begin{equation}\label{eq0}Z_i = \tilde{Z}_i - Y_i\end{equation} is a holomorphic vector field.  The idea is to study the flow generated by the real part of $-Z_i$.

For any point $q\in\partial B(p_\infty, 1)$, take a tangent cone $(V, o')$. Cheeger-Colding \cite{[CC1]} says $(V, o')$ is a metric cone. According to \cite{[CC1]}\cite{[CCT]}, $V$ splits off $\mathbb{R}^2$. 
Take $M_i\ni q_i\to q$. Given any $\epsilon>0$, we may take $\delta=\delta(\epsilon, v)$ so small that $d_{GH}(B(q_i, \frac{1}{\epsilon}\delta), B_V(o', \frac{1}{\epsilon}\delta)) <\epsilon\delta$ for all large $i$. 
By lemma $6.15$ of \cite{[CC1]}, we find harmonic functions $h_i$ on $B(q_i, \delta)$ such that \begin{equation}\label{eq16}|h_i(x)- (r_i(x)-r_i(q_i))|\leq\Phi(\delta)\delta.\end{equation} Furthermore, lemma $6.25$ of \cite{[CC1]} says \begin{equation}\label{eq17}\dashint_{B(q_i, \delta)}|\nabla h_i-\nabla r_i|^2<\Phi(\delta).\end{equation}

By the argument in \cite{[L3]}, we may find a holomorphic function $f_i$ in $B(q_i, \delta)$ such that \begin{equation}\label{eq18}|Re f_i(x) - h_i(x)|<\Phi(\epsilon)\delta,  |\nabla(Re f_i(x)- h_i(x))|\leq \Phi(\epsilon)\end{equation}  in $B(q_i, \frac{1}{2}\delta)$. Given a function $w$ on $B(q_i, \delta)$, define a norm $||w|| = (\dashint_{B(q_i,\frac{1}{2}\delta)}|w|^2)^{\frac{1}{2}}$.

Now by the estimates above, for sufficiently large $i$, we have \begin{equation}\label{eq19}\begin{aligned}||\langle Re Z_i, \nabla f_i\rangle -1||&\leq ||\langle Re(Z_i-\tilde{Z}_i), \nabla f_i\rangle|| + ||\langle \nabla u_i, \nabla f_i\rangle -1||\\&\leq \frac{\Phi(\frac{1}{i})}{\delta^{n}}+C(n)|||\nabla u_i-r_i\nabla r_i|||+||\langle r_i\nabla r_i, \nabla f_i\rangle -1||\\&\leq \frac{\Phi(\frac{1}{i})}{\delta^{n}}+||r_i-1||+||\langle r_i\nabla r_i, \nabla f_i-(\nabla r_i-\sqrt{-1}J\nabla r_i)\rangle|\\&\leq \frac{\Phi(\frac{1}{i})}{\delta^{n}}+ 2\delta+10|||\nabla Ref_i - \nabla r_i|||\\&\leq \frac{\Phi(\frac{1}{i})}{\delta^{n}}+ 2\delta+\Phi(\epsilon|n, v)+\Phi(\delta)\\&<\Phi(\epsilon, \delta|n, v).\end{aligned}\end{equation}

Note that $\langle Re Z_i, \nabla f_i\rangle -1$ is holomorphic on $B(q_i, \delta)$. By the mean value inequality, \begin{equation}\label{eq20}|\langle Re Z_i, \nabla f_i\rangle -1|\leq \Phi(\epsilon, \delta|n, v)\end{equation} on $B(q_i, \frac{\delta}{4})$. Let $\sigma_t$ be the flow generated by $-Re Z_i$. Then \begin{equation}\label{eq21}|\frac{df_i(\sigma_t(q_i))}{dt}+1|\leq \Phi(\epsilon, \delta|n, v)\end{equation} as long as $\sigma_t(q_i)\in B(q_i, \frac{\delta}{4})$.
By applying proposition $6.1$ in \cite{[L3]}, we find $N=N(n, v)$, $\frac{1}{2}>\gamma_1 = \gamma_1(n, v)> 5\gamma_2=5\gamma_2(n, v)>0$, holomorphic functions $g^j_i$ on $B(q_i, \delta)$ such that the following holds:
 $g_i^j(q_i) = 0$; 
  \begin{equation}\label{eq22}\delta=\min\limits_{x\in\partial B(q_i, \frac{\gamma_1\delta}{3})}\sum\limits_{j=1}^N|g_i^j(x)|^2> 2\sup\limits_{x\in B(q_i, \gamma_2\delta)}\sum\limits_{j=1}^N|g_i^j(x)|^2.\end{equation}
 \begin{equation}\label{eq23}\frac{\sup\limits_{x\in B(q_i, \frac{1}{2}\gamma_1\delta)}|g_i^j(x)|^2}{\sup\limits_{x\in B(q_i, \frac{1}{3}\gamma_1\delta)}|g_i^j(x)|^2}\leq C(n, v).\end{equation}
 According to three circle theorem in \cite{[L1]}, $\sup\limits_{x\in\partial B(q_i, \frac{\gamma_1}{2}\delta)}|g_i^j(x)|\leq C(n, v)\delta$. Thus $|dg_i^j|\leq C(n, v)$ on $B(q_i, \frac{5\gamma_1}{12}\delta)$. Then for sufficiently large $i$, \begin{equation}\label{eq24}\dashint_{B(q_i, \frac{5}{12}\gamma_1\delta)}|\langle Re Z_i, dg_i^j\rangle|^2\leq C(n, v)\dashint_{B(q_i, \frac{5}{12}\gamma_1\delta)}|Z_i|^2\leq C(n, v)+\frac{\Phi(\frac{1}{i})}{\delta^{2n}}\leq C(n, v).\end{equation} By similar arguments as above, if $\sigma_t(q_i)\in B(q_i, \frac{1}{3}\gamma_1\delta)$,  \begin{equation}\label{eq25}|\frac{dg_i^j(\sigma_t(q_i))}{dt}|\leq C(n, v).\end{equation} Combining this with (\ref{eq22}), we find $c(n, v)>0$ such that if $|t|\leq c(n, v)\delta$, \begin{equation}\label{eq26}\sigma_t(q_i)\in B(q_i, \frac{1}{3}\gamma_1\delta)\subset B(q_i, \frac{1}{4}\delta).\end{equation} Applying (\ref{eq21}), we find \begin{equation}Ref_i(\sigma_{c(n, v)\delta}(q_i))\leq Ref_i(q_i)-(1-\Phi(\epsilon, \delta|n, v))c(n, v)\delta.\end{equation} If $\epsilon, \delta$ are sufficiently small depending only on $n$ and $v$,  (\ref{eq16}) and (\ref{eq18}) imply
 \begin{equation}\label{eq27}r_i(\sigma_{c(n, v)\delta}(q_i))\leq r_i(q_i)-\frac{1}{2}c(n, v)\delta.\end{equation}  
 We conclude that \begin{equation}\label{eq28}\sigma_{c(n, v)\delta}(B(p_i, 1))\subset B(p_i, 1-\frac{1}{4}c(n, v)\delta).\end{equation}  Also, for any $0<t<c(n, v)\delta$, we may require \begin{equation}\label{eq29}\sigma_t(B(p_i, 1-\frac{1}{4}c(n, v)\delta)\subset B(p_i, 1-\frac{1}{8}c(n, v)\delta).\end{equation}  Now fix $\epsilon=\epsilon(n, v), \delta=\delta(n, v)$ small such that the all inequalities above hold. \begin{prop}\label{prop5}
 There exists a point $o_i$ with $d_i(o_i, p_i)=\Phi(\frac{1}{i})$ and 
 $\lim\limits_{t\to\infty}\sigma_t(B(p_i, 1)) = o_i$. The convergence is uniform on $B(p_i, 1)$. 
 \end{prop}
\begin{proof}
\begin{claim}\label{cl1} $\sigma_t(B(p_i, 1))\subset B(p_i, 1-\frac{1}{8}c(n, v)\delta)$ for all $t\geq c(n, v)\delta$. 
\end{claim}
\begin{proof}
Indeed, we may write $t = kc(n, v)\delta+t'$ where $0\leq t'<c(n, v)\delta$ and $k$ is an integer.
\begin{equation}\label{eq30}\begin{aligned}\sigma_t(B(p_i, 1))&=\sigma_{t'}\sigma_{c(n, v)\delta}\cdot\cdot\sigma_{c(n, v)\delta}(B(p_i, 1))\\&\subset \sigma_{t'}(B(p_i, 1-\frac{1}{4}c(n, v)\delta))\\&\subset B(p_i, 1-\frac{1}{8}c(n, v)\delta).\end{aligned}\end{equation} 
\end{proof}
Recall $B(p_i, 2)$ is contained in the Stein manifold $\Omega_i$. Thus, we may embed $B(p_i, 2)$ in some $\mathbb{C}^{N_i}$. In particular, we have bounded holomorphic functions $z_1, .., z_{N_i}$ which separate points on $B(p_i, 1)$. Consider any sequence $t_j\to\infty$. 
Then by passing to subsequences if necessary, we may assume $z_k(\sigma_{t_j}(x))$ converges uniformly on $B(p_i, 1)$ for all $1\leq k\leq N_i$. Note there is no problem for the uniform convergence close to boundary, according to claim \ref{cl1}. Let $F_i(x) = \lim\limits_{j\to\infty}\sigma_{t_j}(x)$ for $x\in B(p_i, 1)$. Then $F_i$ is a holomorphic map. It is clear that $F_i(B(p_i, 1))\subset B(p_i, 1-\frac{1}{16}c(n, v)\delta)$.

\begin{claim}\label{cl2}
$F_i(B(p_i, 1))$ is a compact subset of $B(p_i, 1-\frac{1}{8}c(n)\delta)$.
\end{claim}
\begin{proof}
It suffices to show that $F_i(B(p_i, 1)) = F_i(B(p_i, 1-\frac{1}{16}c(n, v)\delta)$, since this implies that $F_i(B(p_i, 1)) = F_i(\overline{B(p_i, 1-\frac{1}{16}c(n, v)\delta)})$.
For any $x\in B(p_i, 1)$, $F_i(x) = \lim\limits_{j\to\infty}\sigma_{t_j}(x) = \lim\limits_{j\to\infty}\sigma_{t_{j-1}}(\sigma_{t_j-t_{j-1}}(x))$. We may assume $t_j-t_{j-1}>c(n, v)\delta$ for all $j$. Then $\sigma_{t_j-t_{j-1}}(x)\in B(p_i, 1-\frac{1}{8}c(n, v)\delta)$. By taking subsequence if necessary, we may assume that $\lim\limits_{j\to\infty}\sigma_{t_j-t_{j-1}}(x) = y$.
Since the convergence of $\sigma_{t_j}(x)$ is uniform on $B(p_i, 1)$, $\lim\limits_{j\to\infty}\sigma_{t_{j-1}}(\sigma_{t_j-t_{j-1}}(x)) = F_i(y)\in F_i(B(p_i, 1-\frac{1}{16}c(n, v)\delta)$.
\end{proof}

\begin{claim}\label{cl3}
$F_i(B(p_i, 1))$ is an analytic set in $B(p_i, 1-\frac{1}{16}c(n, v)\delta)$.\end{claim}
\begin{proof}
Since $F_i(B(p_i, 1)) = F_i(\overline{B(p_i, 1-\frac{1}{16}c(n, v)\delta)})$, the claim is a direct consequence of the following proposition, which is the generalization of the proper mapping theorem:
\begin{prop}\label{prop6}\cite{[K]}
Let $M$ and $N$ be connected complex manifolds and $f$ is a holomorphic map from $M$ to $N$. Suppose that for any compact set $L\subset N$, there exists a compact set $K\subset M$ with $L\cap f(M)\subset f(K)$, then $f(M)$ is an analytic set in $N$.
\end{prop}
\begin{remark}
We thank Professor Yum-Tong Siu for pointing this result.
\end{remark}
\end{proof}

Since $B(p_i, 1)$ is connected, $F_i(B(p_i, 1))$ is a connected compact analytic set which is contained in a Stein manifold. Thus it must be a point. Let us say $F_i(B(p_i, 1)) = o_i$. That is, $\sigma_{t_j}$ converges uniformly to $o_i$ on $B(p_i, 1)$. Pick an arbitrary sequence $t'_j\to \infty$, by passing to subsequence if necessary, we may assume that $t'_j>2t_j$ for all $j$. Then $\lim\limits_{j\to\infty}\sigma_{t'_j}(x) = \lim\limits_{j\to\infty}\sigma_{t_j}(\sigma_{t'_j-t_j}(x)) = o_i$. This proves that $\lim\limits_{t\to\infty}\sigma_{t}(B(p_i, 1)) = o_i$ and the convergence is uniform.
Finally, given any $\rho>0$, by using the same argument as before, we can prove that for sufficiently large $i$, $F_i(\overline{B(p_i, \rho)})\subset B(p_i, \rho)$. Thus $o_i\in B(p_i, \rho)$ for sufficiently large $i$. This implies $d_i(o_i, p_i) = \Phi(\frac{1}{i})$.

Therefore, $(M_i, p_i)$ satisfies proposition \ref{prop4} for large $i$. This is a contradiction.
The proof of proposition \ref{prop4} is complete.

\end{proof}

By using the argument in \cite{[L4]} and a rescaling, we obtain the following result.
\begin{cor}\label{cor1}
Give any $n\in\mathbb{N}$ and $v>0$, there exist $\epsilon=\epsilon(n, v)>0, \delta=\delta(n, v)>0$ so that the following holds:
Let $(M^n, p)$ be a complete K\"ahler manifold with $BK\geq-\epsilon$. If $vol(B(p,1))\geq v$ and $d_{GH}(B(p, 1), B_W(o, 1))<\epsilon$ for some metric cone $(W, o)$, then there exists a holomorphic vector field $Z$ on some open set $U\supset B(p, 2\delta)$ so that 
the flow $\sigma_t$ generated by $-Z$ retracts to a point $\tilde{p}$ where $d(p, \tilde{p})=\Phi(\epsilon|n, v)\delta$.  Furthermore, $\sigma_t(B(p, \delta))\subset B(p, 2\delta)$ for all $t\geq 0$.\end{cor}
The next result is suggested by Nan Li. This should be compared with a result in \cite{[GP]}\cite{[GPW]}. See for example, page $206$ and $212$ of \cite{[GPW]}.
\begin{cor}(Uniform contractibility)\label{cor2}
Let $M^n$ be a compact K\"ahler manifold with $BK\geq -1$ and $Vol(M)\geq v$, $diam(M)\leq d$. Then there exists $r_0=r_0(n, v, d)>0$, $C=C(n, v, d)$ so that for any $r<r_0$, $p\in M$, $B(p, r)$ is contractible in $B(p, Cr)$.
\end{cor}
\begin{proof}
The manifold is noncollapsed uniformly. According to Cheeger-Colding \cite{[CC1]} and volume comparison theorem, given any $\delta'>0$, $\epsilon'>0$, we can find $N=N(\delta', \epsilon', n, v, d)$, $r_0=r_0(\delta', \epsilon', n, v, d)$ so that for any $r<r_0$, there exists some integer $m$ between $1$ and $N$ and $d_{GH}(B(p, \frac{r}{\delta'^m})\backslash B(p, \frac{r}{\delta'^{m-1}}), B_V(o, \frac{r}{\delta'^{m}})\backslash B_V(o, \frac{r}{\delta'^{m-1}}))<\epsilon' \frac{r}{\delta'^m}$, where $(V, o)$ is a metric cone. This implies that $d_{GH}(B(p, \frac{r}{\delta'^m}), B_V(o, \frac{r}{\delta'^{m}}))<(\epsilon'+100\delta') \frac{r}{\delta'^m}$. 
Define $(M', g', p') = (M, (\frac{\delta'^m}{r})^2g, p)$. Thus \begin{equation}\label{eq31}d_{GH}(B(p', 1), B_V(o, 1))<\epsilon'+100\delta'\end{equation}
\begin{equation}\label{eq32}BK(M')\geq -\frac{r^2}{\delta'^{2m}},\end{equation}
\begin{equation}\label{eq33}Vol(B(p' ,1)>c(n, v, d)>0. \end{equation}
Let $\delta=\delta(n, v, d)$ be the constant in corollary \ref{cor1} (we have to use the volume lower bound as in (\ref{eq33})).
Now fix $\epsilon'=\epsilon'(n, v, d), \delta'=\delta'(n, v, d)<\delta(n, v, d)$ be sufficiently small so that the right hand side of (\ref{eq31}) is sufficiently small. Then fix $r_0$ be sufficiently small so that the right hand side of (\ref{eq32}) is sufficiently small. We may assume the condition of corollary \ref{cor1} is satisfied. Therefore, $B(p', \delta)$ is contractible in $B(p', 2\delta)$. In particular, $B(p, r)\subset B(p, \delta\frac{r}{\delta'^m})$ is contractible in $B(p, \frac{r}{\delta'^m})\subset B(p, \frac{r}{\delta'^N})$.
This concludes the proof of corollary \ref{cor2}.

\end{proof}

\section{\bf{Some applications to K\"ahler manifolds with nonnegative bisectional curvature}}

Let $(M, p)$ is a complete noncompact K\"ahler manifold with nonnegative bisectional curvature and maximal volume growth. Define $(M_i, p_i, g_i) = (M, p, \frac{g}{r_i^2})$ where $r_i\to\infty$. Cheeger-Colding theory says $(M_i, p_i)$ is getting closer and closer to metric cones. We may apply proposition \ref{prop4} to $(M_i, p_i)$. Let us say the holomorphic vector field is $Z_i$ and the flow $\sigma_t$ generated by $-Re Z_i$ converges to $o_i$.
Since $M_i$ is smooth K\"ahler manifold, at each point $o_i$, we may take a holomorphic chart on $B(o_i, \rho_i)$ which is also diffeomorphic to an Euclidean ball. We may assume $\lim\limits_{i\to\infty}\rho_i = 0$.
When $t$ is sufficiently large, $\sigma_t(B(p_i, 1))\subset B(o_i, \rho_i)$. This proves that $B(p_i, 1)$ is biholomorphic to a domain in $\mathbb{C}^n$.

\begin{claim}\label{cl4}
There exists some open set $B(p_i, \frac{1}{2})\subset U \subset B(p_i, \frac{3}{4})$ such that $U$ is diffeomorphic to $\mathbb{R}^{2n}$.
\end{claim}
\begin{proof}
Now we consider the inverse flow $\sigma_{-t}$. The hope is that for some large $t$, $\sigma_{-t}(B(o_i, \rho_i))$ would be the desired open set. However, in general, this might not be true. The problem is that some point might touch the boundary much earlier than other points. To overcome this difficulty, we cut off the holomorphic vector field $ReZ_i$. More precisely, let $f(u)$ be a smooth function with $f(u)= 1$ for $0\leq u\leq \frac{1}{2}$; $0< f(u)\leq 1$ for $\frac{1}{2}\leq u< \frac{3}{4}$; $f(u) = 0$ for $u\geq\frac{3}{4}$. Let $\sigma'_{-t}$ be the flow generated by $ReZ_i(x)f(r_i(x))$ ($r_i$ is the distance to $p_i$). Then $\sigma'_{-t}$ is not holomorphic, but it induces diffeomorphism. 
Since $d_i(p_i, o_i) = \Phi(\frac{1}{i})$ and $\rho_i\to 0$, $B(o_i, \rho_i)\subset B(p_i, \frac{1}{10})$ for large $i$.
Then $\sigma'_{-t}$ exists for all $t>0$ on $B(o_i, \rho_i)$ and $\sigma'_{-t}(B(o_i, \rho_i))\subset B(p_i, \frac{3}{4})$ for all $t>0$. Each orbit of $\sigma'_{-t}$ belongs to an orbit of $\sigma_{-t}$. According to the uniform convergence of $\sigma_t$,
there exists some $T>0$ such that if $t>T$, $\sigma'_{-t}(\partial B(o_i, \rho_i))\cap B(p_i, \frac{5}{8})=\emptyset$. Let $U=\sigma'_{-2T}(B(o_i, \rho_i))$. Then $U$ satisfies the claim.

\end{proof}

We may pull the open set $U$ back to the original manifold $M$. We obtain an exhaustion of $M$ by Euclidean balls. According to a theorem of Stallings, $M$ is homeomorphic to $\mathbb{R}^{2n}$. If $n\neq 2$, then $M$ is diffeomorphic to $\mathbb{R}^{2n}$. According to \cite{[L3]}, $M$ is biholomorphic to an affine algebraic variety. If $n=2$, by a theorem of Ramanujam, $M$ is biholomorphic to $\mathbb{C}^2$. we postpone the proof for general $n$ to the last section.

\section{\bf{Proof of theorem \ref{thm2}}}

\begin{proof}
According to the main theorem in \cite{[L4]}, $X$ is homeomorphic to an irreducible normal complex analytic space. Take $q\in X$ and a tangent cone $V$ at $q$. Let $\epsilon$ be a very small number. Then there exists $r>0$ such that $d_{GH}(B_X(q, 100r), B_V(o, 100r))<\epsilon r$. We may assume $r$ is sufficiently small.
First we give a separate proof for the case when $n=2$, since the argument is easier and more instructive. Then as $X$ is normal with dimension $2$, the possible singularities are all isolated. Without loss of generality, assume $q$ is an analytic singular point and a small punctured ball $B(q, 100r)\backslash\{q\}$ is analytically smooth.  Take a closed curve $\gamma$ in the small punctured ball $B(q, r)\backslash\{q\}$. We may assume that there exists some $\epsilon_1>0$ so that $\gamma\subset B(q, r)\backslash B(q, \epsilon_1r)$.
\begin{lemma}
$\gamma$ is contractible in $B(q, 10r)\backslash\{q\}$, if $\epsilon$ is small enough.
\end{lemma}
\begin{proof}
Consider $M_i\ni q_i\to q$ in the Gromov-Hausdorff sense. Then according to the argument in \cite{[L4]}, we may assume that the Gromov-Hausdorff convergence is in fact smooth in the complex analytic sense (not necessarily in metric sense) from $B(q_i, 50r)\backslash B(q_i, \frac{1}{100}\epsilon_1r)$ to $B(q, 50r)\backslash B(q, \frac{1}{100}\epsilon_1r)$. In particular, there exists a diffeomorphism from $\gamma\subset B(q, r)\backslash B(q, \epsilon_1r)$ to $U\subset\subset B(q_i, 50r)\backslash B(q_i, \frac{1}{100}\epsilon_1r)$.
We may life the curve $\gamma$ to $U$. It suffices to prove the image of $\gamma$ is contractible on $U$. This can be done by using the same argument as in claim \ref{cl4}. Basically we prove that the image of $\gamma$ lies between two topological balls. The details are omitted.

\end{proof}

Now we apply Mumford criteria \cite{[Mu]} to obtain that $q$ is in fact a smooth point for the normal variety $X$. Thus $X$ is complex analytically smooth. Let $z_1, z_2$ be a holomorphic chart around $q\in X$. According to lemma $7.1$ in \cite{[L4]}, we may find holomorphic functions $z_1^i, z_2^i$ on fixed size neighborhood of $q_i$ with $z_1^i\to z_1$, $z_2^i\to z_2$. By using a degree argument, one can verify that $z_1^i, z_2^i$ form a holomorphic chart on some fixed size neighborhood of $q_i$. 
The stability for dimension $2$ follows from a standard gluing argument.

\medskip

Next we consider the general $n$ dimensional case. Let $(V, o)$ be a tangent cone at $q$. Take a sequence $r_i\to 0$ such that the rescaled metrics $(M'_i, q'_i) = (M_i, q_i, \frac{g_i}{r_i^2})\to (V, o)$ in the pointed Gromov-Hausdorff sense. By corollary \ref{cor1}, for sufficiently large $i$, we can define a holomorphic vector field $Z_i$ on $B(q'_i, 100)$. By shifting $q'_i$ a little bit, we may assume the flow generated by $-Re Z_i$ converges to $q_i'$.
\begin{prop}\label{prop7}
Let $\sigma^i_t$ be the flow generated by $-Re Z_i$, $\rho_t$ be the flow on $(V, o)$ generated by $-r_V\frac{\partial}{\partial r_V}$ where $r_V$ is the distance to $o$. Then $\sigma^i_t$ converges to $\rho_t$ uniformly.  More precisely, if $y_i\in B(q'_i, 100)$ and $y_i\to y\in V$, then $\sigma^i_t(y_i)\to \rho_t(y)$ for any $t>0$.
\end{prop}
\begin{proof}
We need two lemmas.

\begin{lemma}\label{lm1}
Let $x$ be a regular point (in metric sense) around $o\in V$. Take a sequence $M'_i\ni x_i\to x$. Then $\sigma^i_t(x_i)$ converges to $\rho_t(x)$.
\end{lemma}
\begin{proof}
For simplicity, we assume $d(x, o) = 1$. The general case easily follows from a rescaling argument.
Let $C$ be a large constant, to be determined. Take a small geodesic ball $B(x, Cr)$ such that we have a holomorphic chart $U=(z_1, ..., z_n)$. According to lemma $7.1$ in \cite{[L4]}, if $C$ is large enough, we also have holomorphic charts $U_i=(z^i_1, ..., z^i_n)$ around $B(x_i, r)$ and $z^i_j\to z_j$. In view of (\ref{eq24}), by passing to subsequence if necessary, we may assume the holomorphic vector field $Z_i$ converges to a holomorphic vector field $Z$ on $U$. Then $\sigma^i_t$  converges on $B(x_i, \frac{r}{10})$ to a holomorphic map $\sigma_t$ for $|t|$ small.
\begin{claim}\label{cl5} For any sequence $t_k\to 0$, $\lim\limits_{k\to\infty}\frac{z_j(\sigma_{t_k}(x))-z_j(\rho_{t_k}(x))}{t_k} = 0$.
\end{claim}
\begin{proof}
We first blow up $x\in V$. Then $(V_k, x''_k, \tilde{d}_k) = (V, x, \frac{d}{t_k})\to (\mathbb{R}^{2n}, 0)$. Let $\Phi_k$ be the Gromov-Hausdorff approximation from $(M'_k, q'_k)$ to $(V, o)$. Below we shall pass to subsequence of $M'_k$ which is still denoted by $M'_k$. That is to say, for each $k$, $M'_k$ is arbitrarily close to $V$ as we want.  Let us say $\Phi_k$ is a $t_k$-Gromov-Hausdorff approximation from $B_{N_k}(x'_k, \frac{1}{t_k})$ to $B(x''_k, \frac{1}{t_k})$, where $(N_k, x'_k, d'_k)=(M'_k, x_k, \frac{d_k}{t_k})$.  As before, consider the holomorphic coordinates $(w^k_1, ..., w^k_n)$ around $x'_k$ satisfying
$w^k_s(x'_k) = 0$, $(w^k_1, ..., w^k_n)$ is a $\Phi(\frac{1}{k})$-GH approximation to the Euclidean ball $B(0, 100)$. We may further assume that $w^k_1 = f_k$ as constructed in (\ref{eq18}). We may regard $Z_k$ as a holomorphic vector field on $N_k$. Define $Z'_k = t_kZ_k$. Let $\sigma'_t$ be the flow generated by $-Re Z'_k$ on $(N_k, x'_k)$.
We have \begin{equation}\label{eq34}\dashint_{B(x'_k, 100)}|\langle dw^k_i, \overline{dw^k_j}\rangle-2\delta_{ij}|^2<\Phi(\frac{1}{k}).\end{equation} As $d(x, o) = 1$, by using the same argument as in (\ref{eq19}),  \begin{equation}\label{eq35}|\langle Re Z'_k, dw^k_i\rangle-\delta_{i1}|<\Phi(\frac{1}{k}).\end{equation} Therefore \begin{equation}\label{eq36}|w^k_i(\sigma'_1(x'_k))+\delta_{i1}|<\Phi(\frac{1}{k}).\end{equation} On the other hand, we may require that for any $z\in B(x_k', 2)$, \begin{equation}\label{eq37}|Re w^k_1(z) - \frac{d_{M'_k}(q'_k, z)}{t_k}+\frac{d(o, x)}{t_k}|<\Phi(\frac{1}{k}).\end{equation} By definition of $\rho_t$, 
\begin{equation}\label{eq38}|w^k_i(\overline{\rho_{t_k}(x))}+\delta_{i1}|<\Phi(\frac{1}{k}),\end{equation} where $\overline{\rho_{t_k}(x)}$ is a preimage of $\rho_{t_k}(x)$ on $N_k$ under $\Phi_k$.  We have \begin{equation}\label{eq39}|w^k_i(\sigma'_1(x'_k))-w^k_i(\overline{\rho_{t_k}(x)})|<\Phi(\frac{1}{k}).\end{equation} The gradient estimate says $|dz^k_j|\leq C$ on $B(x_k, \frac{r}{10})$, where $C$ is independent of $k, j$.  By Cauchy estimate, $|\frac{\partial z^k_j}{\partial w^k_s}|\leq Ct_k$. Then \begin{equation}\label{eq40}|z^k_j(\sigma'_1(x'_k))-z^k_j(\overline{\rho_{t_k}(x)})|<t_k\Phi(\frac{1}{k}).\end{equation} If $M_k'$ is sufficiently close to $V$, we can ensure that \begin{equation}\label{eq41}\frac{z^k_j(\sigma'_1(x'_k))-z_j(\sigma_{t_k}(x))}{t_k}=\Phi(\frac{1}{k}),  \frac{z^k_j(\overline{\rho_{t_k}(x)})-z_j(\rho_{t_k}(x))}{t_k}=\Phi(\frac{1}{k}).\end{equation} (\ref{eq40}) and (\ref{eq41}) give the proof of claim \ref{cl5}.

\end{proof}
Observe that the argument in claim \ref{cl5} works for any regular point. Since $\rho_t(x)$ are all regular points for $t>0$, $\frac{d\rho_t(x)}{dt} = (-Re Z)(\rho_t(x))$ for all $t$. By definition, $\frac{d\sigma_t(y)}{dt} = (-Re Z)(\sigma_t(y))$ for all $t, y$. Then by the uniqueness of the integral curve generated by $-Re Z$, we find $\sigma_t(x) = \rho_t(x)$ for all small $t>0$. This completes the proof of lemma \ref{lm1}.

\end{proof}
For later purposes, let us note the following
\begin{cor}\label{cor00}
Let $f_i$ be a sequence of holomorphic functions on $B(q_i', 10)$ so that $f_i\to f_\infty$ uniformly on each compact set. Then $(-Re Z_i)(f_i)\to (-r_V\frac{\partial}{\partial r_V})f_\infty$ uniformly on $B(q_i', 2)$.
\end{cor}
\begin{proof}
Claim \ref{cl5} says we have the convergence on the regular points on the limit space. Note that $|(-Re Z_i)(f_i)|$ has uniform $L^2$ bounds on $B(q_i', 5)$, as $Z_i$ has uniform $L^2$ bound. Mean value inequality gives the uniform bounds for $|(-Re Z_i)(f_i)|$, hence its gradient bound. As regular points are dense on limit space, this completes the proof.
\end{proof}

\begin{lemma}\label{lm2}Let $y_i\in M'_i$, $y\in V$ and $y_i\to y$.
Then there exist $N>0, r>0, r'>0$ which are independent of $i$ such that for all sufficiently large $i$, the following holds:

1. there exist a sequence of holomorphic embeddings $F_i$: $B(y, r')\to \mathbb{C}^N$,  $F:$ $B(y, r')\to \mathbb{C}^N$.

2. $F_i\to F$, if we compose with the Gromov-Hausdorff approximation. Thus, the image of $B(y_i, r')$ converges to the image of $B(y, r')$ in the Hausdorff topology of $\mathbb{C}^N$.

3. $F_i(y_i) = F(y) = 0\in\mathbb{C}^N$.

4. $B(y_i, r)\subset F_i^{-1}(B_{\mathbb{C}^N}(0, 1))\subset\subset B(y_i, r')$ and $B(y, r)\subset F^{-1}(B_{\mathbb{C}^N}(0, 1))\subset\subset B(y, r')$.\end{lemma}
\begin{proof}
According to the main theorem in \cite{[L4]}, $(V, o)$ is a normal complex analytic space. Thus for some small $a, b>0$ and some large $C>0$, we may holomorphically embed $B(y, Ca)$ in some $(\mathbb{C}^{N'}, 0)$ by map $F'$ ($Ca$ is still small). Say $F'(y) = 0\in\mathbb{C}^{N'}$. Also we may assume $B(y, 2b)\subset F'^{-1}(B(0, 1))\subset\subset B(y, a)$. Let the coordinates on $\mathbb{C}^{N'}$ be $z_1, ..., z_{N'}$. According to lemma $7.1$ in \cite{[L4]}, we may find holomorphic functions $z^i_1, ..., z^i_{N'}$ on $B(y_i, 2a)$ so that $z^i_s\to z_s$ for all $s$.  By shrinking $a$ if necessary, we may assume for each $B(y_i, 2a)$, there exists a holomorphic flow $\sigma^i_t$ retracting to $y_i^{''}$. Consider the local holomorphic coordinates $U_i=(w^i_1, ..., w^i_n)$ near $y_i^{''}$ with $w^i_j(y_i^{''}) = 0$ for all $j$. By scaling, we may assume $|w^i_j|\leq \frac{1}{2^i}$ for all $i, j$. Note that the size of $U_i$ could go to zero. For each $i$, there exists large $t_i$ so that $\sigma^i_{t_i}(B(y_i, 2a))\subset U_i$. We can pull the coordinate back to $B(y_i, 2a)$ via $\sigma^i_{t_i}$. Say the new coordinate on $B(y_i, 2a)$ is still denoted by $(w^i_1, .., w^i_n)$. Take $r=b$, $r'=2a$, $N=N'+n$. One can verify that the holomorphic maps $F_i = (z^i_1, ..., z^i_{N'}, w^i_1, .., w^i_n) $ and $F=(z_1, ..., z_{N'}, 0, 0, 0, .., 0)$ satisfy the lemma.
\end{proof}
Let $z_1, .., z_N$ be coordinates of $\mathbb{C}^N$. We identify $B(y, r)$ and $B(y_i, r)$ with their images in $\mathbb{C}^N$.
By (\ref{eq26}), for each $t$ small, $\sigma^i_t$ is bounded. Thus $z_j(\sigma^i_t)$ is uniformly bounded for $j = 1, ..., N$. We can extract a subsequence such that $z_j(\sigma^i_t)$ all converge uniformly on $B(y_i, r)$.
Lemma \ref{lm2} says that $\sigma^i_t(B(y_i, r))$ converges uniformly. Notice $y$ is arbitrary. Since regular points are dense, by lemma \ref{lm1}, 
we conclude the proof of proposition \ref{prop7}.
\end{proof}

From proposition \ref{prop7}, we see $-r_V\frac{\partial}{\partial r_V}$ is a holomorphic vector field on $V$. Given any holomorphic function $f$ around $o\in V$, we may write $f$ as an infinite sum of homogeneous harmonic functions (basically we just do the spectral decomposition on the cross section). We claim that each homogeneous function appeared must be holomorphic. For instance, to show the lowest degree harmonic function (say degree $a$) is holomorphic, one verifies that it is the limit of $f(\sigma_t(x))e^{at}$ as $t\to+\infty$.
By subtracting the first function, one can show the remaining homogeneous harmonic functions are all holomorphic.

Let $z_1, ..., z_{N'}$ be holomorphic functions on $V$ which give a local holomorphic embedding near $o$.  Let us say $z_s(o)=0$ for all $s$. Now we use some argument in \cite{[DS2]}.
Consider the restriction of $z_1, ..., z_{N'}$ on $V$. We write $z_s = \sum z^\alpha_s$ where $z^\alpha_s$ are all homogeneous holomorphic functions as in the last paragraph. Then we extend each $z^\alpha_s$ to a holomorphic function on $B(0, 1)$ of $\mathbb{C}^N$. 
We may require that the sum is still equal to $z_s$ on $B(0, \frac{1}{2})$, since there is a bounded extension of holomorphic functions. See for example, corollary $4$ on page $157$ of \cite{[GR]}. Then for each $s$, we can find some $z^\alpha_s$ so that $\det(\frac{\partial z^{\alpha}_s}{\partial z_s})\neq 0$ at $0$. According to the implicit function theorem, these $z^\alpha_s$ form a local holomorphic chart in a small neighborhood of $0$ in $\mathbb{C}^N$.  Thus we obtain a global holomorphic embedding from $V$ to $\mathbb{C}^{N'}$ by these $z^\alpha_s$. For notational convenience, we still denote the homogeneous coordinates by $z_s$. Consider the integral ring $R$ generated by functions $z_s$ on $V$. By using the three circle theorem in \cite{[L1]}, we can prove the dimension estimate $dim(\mathcal{O}_d(V))\leq Cd^n$ as in the smooth case. Here $\mathcal{O}_d(V)$ denotes polynomial growth holomorphic functions with degree bounded by $d$. By a dimension counting argument, we see that the affine algebraic variety defined by $R$ has dimension $n$. Then we can verify that $V$ is biholomorphic to the affine algebraic variety defined by $R$. Since the argument is very similar to section $7$ in \cite{[L3]}, we skip the details.

We may find $C>0$ so that $(z_1, .., z_{N'})^{-1}B_{\mathbb{C}^N}(0, 10)\subset\subset B(o, C)$. Moreover, $(z_1, .., z_{N'})$ is a holomorphic embedding on $B(o, 2C)$.
Next we lift these $z_s$ to $M'_i$, say $z^i_s\to z_s$ uniformly on $B(q'_i, C)$. We add coordinate functions $w^i_k (k=1, .., n)$ as introduced in the proof of lemma \ref{lm2}. Set $N=N'
+n$. Then we can define embeddings $F_i\to F$ as $(z^i_1, .., z^i_{N'}, w^i_1, ..., w^i_n)$ and $(z_1, .., z_{N'}, 0, 0, .., 0)$. We identify $V$ and $B(q'_i, C)\subset M'_i$ with their image in $\mathbb{C}^{N}$. Let $(z_1, .., z_N)$ be the coordinate on $\mathbb{C}^N$.
Define a holomorphic vector field $Y$ on $\mathbb{C}^{N}$ by $Y=-\sum_j\alpha_jz_j\frac{\partial}{\partial z_j}$ where $\alpha_j\geq 0$ is the degree of $z_j$. For $j>N'$, we set $\alpha_j=0$. Then $Re Y$ coincides with the vector field $-r_V\frac{\partial}{\partial r_V}$ on $V$. Observe on the intersection of the unit sphere in $\mathbb{C}^N$ and $V$, \begin{equation}\label{eq42}Re Y(\sum |z_j|^2)<0.\end{equation} Note also that $\alpha_j\geq 1$ if $z_j$ is not zero. Otherwise, there exists some homogeneous function which is of sublinear growth. Then by gradient estimate, it must be constant, hence, identically zero.
Since $-Re Z_i(z_j)$ (holomorphic) is uniformly convergent to $-Re Y(z_j)$ on $B(o, \frac{C}{2})$, we have for sufficiently large $i$, \begin{equation}\label{eq-1}-Re Z_i(\sum |z_j|^2)<-\frac{1}{10}\end{equation} on the intersection of the unit sphere in $\mathbb{C}^N$ and $M'_i$. We have used $\alpha_j\geq 1$.

\begin{prop}\label{prop8}Given any $n\in\mathbb{N}, v>0$, there exist $\epsilon = \epsilon(n, v)>0, \delta=\delta(n, v)>0$ so that the following holds:
let $(M^n, p)$ be a complete K\"ahler manifold with bisectional curvature lower bound $-\epsilon^3$. Assume $vol(B(p, \frac{1}{\epsilon}))\geq \frac{v}{\epsilon^{2n}}>0$ and $B(p, \frac{1}{\epsilon})$ is $\epsilon$-Gromov-Hausdorff close to a geodesic ball centered at the vertex of $(\mathbb{R}^{2k}, 0)\times (Y, o)$.  Here $(Y, o)$ is a metric cone. Then there exists an open set $U$ and a map $F$ satisfying $B(p, \delta)\subset U\subset B(p, 1)$, $F: \mathbb{D}^k\times Z\to U$ is a biholomorphism. Here $\mathbb{D}^k$ is a polydisk in $\mathbb{C}^k$ and $Z$ is a complex manifold. Finally, for any two points $y_1, y_2\in B(p, \delta)\cap F((0, 0, .., 0)\times Z)$, there exists a curve $l\subset B(p, g(\delta))\cap F((0, 0, .., 0)\times Z)$ connecting $y_1, y_2$. Here $g$ is a continuous increasing function on $\mathbb{R}$ depending only on $n, v$ and $g(0) = 0$. Also, $g(\delta)<\frac{1}{10}$.
\end{prop}
\begin{proof}
We will assume $\epsilon$ is as small as we want. Eventually we see its value depends only on $n, v$.
By proposition \ref{prop1} and similar arguments in section $3$, we find holomorphic functions $z_1, .., z_k$ and holomorphic vector fields $X_1, .., X_k$ on $B(p, 1)$ ($X_s$ is obtained by perturbing the gradient of the harmonic functions) such that $(z_1, .., z_k)$ almost gives the splitting of the factor $\mathbb{R}^{2k}$.  
Also, \begin{equation}\label{eq43}|X_i(z_j)-\delta_{ij}|<\Phi(\epsilon|n, v).\end{equation} Let us assume $z_s(p) = 0$ for all $s$. Now we use induction. When $k=0$, the conclusion is trivial. Assume the proposition is true for $k=m-1$ and $U_{m-1}=\mathbb{D}^{m-1}\times Z_{m-1}$ where $Z_{m-1}$ is given by the zeros of $z_1, ..., z_{m-1}$.  We shall prove it $k=m$. Let $\epsilon_{m-1}, \delta_{m-1}$ be the corresponding constants in proposition \ref{prop8} for $k=m-1$.
\begin{claim}\label{cl6}
There exist holomorphic functions $c_1, ..., c_{m-1}, c_m$ with $|c_i|(1\leq i\leq m-1)$ and $|1-c_m|$ very small such that if $W_m = c_mX_m-\sum\limits_{i=1}^{m-1}c_iX_i$, then $W_m(z_j)=\delta_{jm}$.
\end{claim}
\begin{proof}
Just use linear algebra, in view of (\ref{eq43}).
\end{proof}
The claim says $W_m$ is a holomorphic vector field on $Z_{m-1}$. Define $Z_m$ be the zeros of $z_1, ..., z_m$. Then $Z_m$ is smooth by claim \ref{cl6}. Let $\sigma_t$ be the flow on $Z_{m-1}$ generated by $W_m$ for $t\in\mathbb{C}$. Define a holomorphic map $\sigma: Z_m\times \Delta_m$ ($\Delta_m = \{t\in\mathbb{C}||t|<\gamma\}$) to $Z_{m-1}$ as 
$\sigma(x, t) = \sigma_t(x)$. Here $\gamma=\gamma(n, v)$ is small, to be determined.
\begin{claim}\label{cl7}
If $\gamma$ is small, $\sigma$ defines a biholomorphic map onto its image which contains $B(p, \delta_m)\cap Z_{m-1}$ for some $\delta_m=\delta_m(n, v)>0$.
\end{claim}
\begin{proof}
By using the same argument as in section $3$, we see that if $\gamma=\gamma(v, n)$ is small and $x\in Z_m\cap B(p, \frac{1}{10})$, $\sigma(t, x)\in B(x, \frac{1}{10})$ for $|t|<\gamma$.
By claim \ref{cl6}, $z_j(\sigma(x, t)) = t\delta_{jm}$. Therefore $\sigma(x, t)$ is injective for $x\in Z_m\cap B(p, \frac{1}{10})$, $|t|<\gamma$. If $\delta_m$ small and $y\in B(p, \delta_m)\cap Z_{m-1}$, $\sigma(y, -z_m(y))\in Z_m$. \end{proof}
Let $y_1, y_2\in B(p, \delta_m)\cap Z_{m}$. We may assume $\delta_m<\delta_{m-1}$. By induction hypothesis, there exists a curve $l\subset B(p, g_{m-1}(\delta_m))\cap Z_{m-1}$ connecting $y_1, y_2$. Here $g_{m-1}$ is a continuous increasing function with $g_{m-1}(0)= 0$. If $\delta_m=\delta_m(n, v)$ is small enough, we can project $l$ to $Z_m$ via $\sigma$. The image lies in $B(p, g_m(\delta_m))$ for some continuous increasing function $g_m$ depending only on $n, v$. Of course, we may assume $g_m(\delta_m)<\frac{1}{10}$. The last assertion is verified.  The proof of proposition \ref{prop8} is complete.
\end{proof}
\begin{remark}
In the proposition above, we have identified $\mathbb{D}^k = \Delta_1\times \Delta_2\times\cdot\cdot\times \Delta_k$ where $\Delta_k$ is defined right above claim \ref{cl7}.
Let us endow $\mathbb{D}^k$ with the product metric on the right hand side. Let the distance on $Z_m$ be induced by the distance on $M$. Then by a limiting argument, we can prove that the biholomorphic map $F$ is a $\Phi(\epsilon|n, v)$-Gromov-Hausdorff approximation. 
\end{remark}
\begin{prop}\label{prop9}
Let $(M^n_i, p_i)$ be a sequence of pointed K\"ahler manifolds converging to $X=\mathbb{C}^{k}\times (V, o)$ in the Gromov-Hausdorff sense, where $(V, o)$ is a metric cone. Assume the bisectional curvature of $M^n_i$ has lower bound $-\Phi(\frac{1}{i})$ and $vol(B(p_i, r))\geq cr^{2n}$ for any $0<r<R_i$, where $c$ is a positive constant and $R_i\to\infty$.
Then $V$ is homeomorphic to an irreducible normal complex analytic variety. 
\end{prop}
\begin{proof}
By applying proposition \ref{prop8} to $M_i$, we construct holomorphic vector fields $W^i_j(1\leq j\leq k)$ and holomorphic functions $z^i_s$ so that $W^i_j(z^i_s) = \delta_{js}$. 
Let $\sigma_{ij}(t)$ be the biholomorphisms induced by $W^i_j$. Then after passing to subsequences, $\sigma_{ij}(t)\to\sigma_{j}(t)$ which induces biholomorphism on the limit space. Also $z^i_s\to z_s$. Now set $\Sigma$ be the zero set of $z_1, .., z_k$. We should regard $\Sigma$ as a closed subscheme induced by the ideal generated by $z_1, .., z_k$. Since $X$ is irreducible, by using projections as in the last part of the proof of proposition \ref{prop8}, we see that the regular points of $\Sigma$ are connected. By claim \ref{cl6}, we see $\Sigma$ is reduced. Therefore $\Sigma$ is integral. We can also verify that $X$ is isomorphic to $\Sigma\times\mathbb{C}^k$ as a complex space. Since $X$ is normal, $\Sigma$ must be normal. Also $X$ is isometric to $\Sigma\times\mathbb{C}^k$ (the metric on $\Sigma$ is induced from $X$), since the coordinate functions $z_1, ..., z_k$ are Euclidean splitting factors.

\end{proof}
\begin{prop}\label{prop10}
In proposition \ref{prop9}, If $k\geq n-3$, then $X$ is in fact a complex manifold.\end{prop}
\begin{proof}
If $k=n$ or $n-1$, the conclusion follows from sec $6$ of \cite{[L4]}. 
Let $Z_i$ be the holomorphic vector field in (\ref{eq0}). By shifting $p_i$ a little bit if necessary, we may assume that the flow generated by $-Re Z_i$ converges to $p_i$.

\emph{Hypothesis}: $V$ is complex analytically smooth away from $o$.

We first assume the hypothesis above.
According to the analysis right above proposition \ref{prop8}, we find local embeddings of $(M_i, p_i)$ and $X$ to $\mathbb{C}^N$. We have assumed all coordinate functions are homogeneous on $X$. Also the embedding maps $p_i$ to the origin of $\mathbb{C}^N$. Furthermore, $M_i$ (local part containing $p_i$) converges to $X$ in the Hausdorff topology in $\mathbb{C}^N$. Now we use the same notations as in proposition \ref{prop9}. Since the holomorphic functions $z^i_s$ converge to $z_s$, the Hausdorff limit of $Z^i_k$ is contained in $\Sigma$.  Here $Z^i_k$ is the common zeroes of $z^i_1, .., z^i_k$. For simplicity, we may assume $z_s (1\leq s\leq k)$ are contained in the coordinate functions of $\mathbb{C}^N$ (if not, we just add them to the coordinate. Notice that these $z_s$ are homogeneous of degree $1$ on $X$). On the other hand, given any point on $\Sigma$, we first find a point on $M_i$ which is very close to that point. Then by using the flow introduced in proposition \ref{prop8}, one can find a nearby point so that the holomorphic functions $z^i_s$ all vanish. This proves that $Z^i_k$ converges to $\Sigma$ in the Hausdorff sense. One also verifies that the limit of $Z^i_k$ has multiplicity $1$ by claim \ref{cl6}. We identify $V$ and $\Sigma$. Now pick a point $p\in \Sigma\backslash 0$. By the hypothesis above, $p$ is a regular point on $\Sigma$. Let $A$ be the intersection of $\Sigma$ and the unit sphere of $\mathbb{C}^{N}$.  Observe that $\Sigma$ is invariant under the flow generated by the real holomorphic vector field $-r\nabla r$ ($r$ is the distance to the vertex $(0, o)\in X=\mathbb{C}^k\times V$). According to (\ref{eq42}), $A$ is transverse to the $-r\nabla r$ on $\Sigma$. Therefore, $A$ is smooth. Let $A_i$ be the intersection of the unit sphere in $\mathbb{C}^N$ and $Z^i_k$ for large $i$. Then $A_i$ is diffeomorphic to $A$ for large $A$, since $A$ is compact and smooth. $A$ and $A_i$ admit natural contact structure induced by the Levi form. For this part, please refer to section $1$ and section $2$ of \cite{[Mc]}.
Due to the stability of contact structure (theorem $2.2$ in \cite{[Mc]}), $A$ is in fact contactomorphic to $A_i$ for sufficiently large $i$.

 Let $W^i_j$ be the holomorphic vector field so that $W^i_j(z^i_s) = \delta_{js}$. The argument is the same as in claim \ref{cl6}.
Define a holomorphic vector field $H_i = Z_i - \sum\limits_{j=1}^kc_{ij}W^i_j$. Here $c_{ij}$ are holomorphic functions so that $H_i$ is tangential to $Z^i_k$. 
\begin{claim}\label{cl8}
$\lim\limits_{i\to\infty}\langle Z_i, dz^i_s\rangle=0$ for $1\leq s\leq k$ on $A_i$. 
\end{claim}
\begin{proof}
Observe $\langle Z_i, dz^i_s\rangle$ is holomorphic. If $i\to\infty$, by passing to subsequence, we have uniform convergence on $A_i$. Notice in the limit case, $\langle r\frac{\partial}{\partial r}, dz_s\rangle = 0$ on $A$. This concludes the proof.
\end{proof}

Claim \ref{cl8} implies that $c_{ij}$ are small functions on $A_i$. Combining this with (\ref{eq-1}), we find $-Re H_i(\sum|z_j|^2)<0$ on $A_i$. On $Z^i_k$, let $\Phi^i_t$ be the biholomorphism generated by $-Re H_i$. Notice that the open set on $Z^i_k$ bounded by $A_i$ is connected: given any two points $a, b$ there, connect them by a shortest geodesic $L$ on $M_i$. We can project $L$ to $Z^i_k$ by using the flow $\sigma$ as in proposition \ref{prop8}. Say the image of $L$ is $L'$. Notice $L'$ is contained in a uniformly bounded set of $Z^i_k$. For any $R>0$, if $i$ is large, by using the same argument as above, we may assume that $-Re H_i(\sum|z_j|^2)<0$ on $(B_E(0, R)\backslash B_E(0, 1))\cap Z^i_k$ ($E$ is the Euclidean metric on $\mathbb{C}^N$).
Thus for large $t$, $\Phi^i_t(L')$ is contained in the domain bounded by $A_i$ on $Z^i_k$. This proves the connectness.
Then by the same argument as in proposition \ref{prop5}, the flow $\Phi^i_t$ converges to a point. According to our assumption in the beginning of proposition \ref{prop10}, the flow generated by $-Re Z_i$ converges to $p_i$ which is $0\in\mathbb{C}^N$. Thus $Z_i$ vanishes at $0$. Therefore, $c_{ij}$ all vanish at $0$.
This implies that $\Phi^i_t$ is retracting to $0$. Now we freeze $i$ for a moment. Let $B_i = \frac{d\Phi^i_t}{dt}|_{t=0}$ on $T_0Z^i_k$. By Schwarz lemma, the real part of the eigenvalues of $B_i$ are all negative. Consider a coordinate $U$ given by $(\tilde{z}_1, ...,\tilde{z}_{n-k})$ around $0$ of $Z^i_k$.  We may assume that $B_i$ has the Jordan normal form.
In particular, $B_i$ is an upper triangular matrix. Then the real part of entries of the main diagonal are all negative. By rescaling each $\tilde{z}_j$ by some positive factors, we may assume that the absolute values of entries off the main diagonal are all very small. We still denote the new coordinate system by $(\tilde{z}_1, ..., \tilde{z}_{n-k})$. Now for any $z\in U$,  define $Q(z, \overline{z}) = \sum|\tilde{z}_j(z)|^2$. For very small $\epsilon>0$, define $D=\{z\in U|Q(z, \overline{z})=\epsilon\}$ which is contactomorphic to the standard sphere $\mathbb{S}^{2(n-k)-1}$. 
\begin{lemma}\label{lm3}
For any $z\in D$, $\frac{dQ(\Phi^i_t(z), \overline{\Phi^i_t(z)})}{dt}|_{t=0}<0$. Therefore, $-Re H_i$ is pointing inside the sphere $D$ transversely.\end{lemma}
\begin{proof}
According to the assumptions of $B_i$, for any $z\in\mathbb{C}^{n-k}$ (we consider $z$ as a column vector), $z^TB^T_i\overline{z} + z^T\overline{B}_i\overline{z}<-\delta|z|^2$ for some $\delta>0$. Note that near the origin, by using the coordinate $(\tilde{z}_1, ..., \tilde{z}_{n-k})$, we have the vector field $H_i(z) = -2B_iz+O(|z|^2)$.  The factor $-2$ comes from the assumption that $\Phi^i_t$ is generated by $-Re H_i$. Therefore \begin{equation}\label{eq44}\frac{dQ(\Phi^i_t(z), \overline{\Phi^i_t(z)})}{dt}|_{t=0} = z^TB^T_i\overline{z} + z^T\overline{B}_i\overline{z} +o(|z|^2)<0.\end{equation}

\end{proof}
Recall that $Re H_i$ is transverse to $A_i$ and inside the domain bounded by $A_i$ on $Z^i_k$, $Re H_i$ is nonvanishing except at $0$. According to lemma \ref{lm3}, $A_i$ and $A$ are diffeomorphic to $\mathbb{S}^{2(n-k)-1}$. 

\medskip
\begin{claim}\label{cl9}
The hypothesis is satisfied for $k=n-2$.
\end{claim}
\begin{proof}
For any point $a\in\Sigma\backslash\{0\}$, take a tangent cone $W$ at $a$. Then $W$ splits off $\mathbb{R}^{2n-2}$.  Then we apply the result for $k=n-1$ to see that $W$ is in fact smooth. We can pull the holomorphic chart back to a small neighborhood of $a$. By a degree argument, we verify that this is a holomorphic chart around $a$. 
\end{proof}

If $k=n-2$, from the Mumford criteria \cite{[Mu]}, we see that $\Sigma$ is in fact smooth. This concludes the proof for $k=n-2$.
For $k=n-3$, we can use the same argument as in claim \ref{cl9} to show that the hypothesis is satisfied.
Next we need a result on the contact structure of boundary of strictly pseudoconvex domains:
\begin{prop}\label{prop11}
Let $V_1, V_2$ be strictly pseudoconvex domains in $\mathbb{C}^n$ with smooth boundary. Assume the closure of $V_2$ is contained in $V_1$.
Let $X$ be a real holomorphic vector field ($X$ induces biholomorphism) defined in a neighborhood of $V_1$ which points inward the boundary of $V_1$ and the boundary of $V_2$.
Assume $X$ has only one zero point $p$ inside $V_2$ and the flow generated by $X$ is retracting to $p$ on $V_1$. Then the boundary of $V_1$ is contactomorphic to the boundary of $V_2$.
\end{prop}
\begin{proof} For notational convenience, we simplify plurisubharmonic function as psh function. It suffices to construct a strictly psh function $f$ without critical point such that the boundary of $V_1$ and boundary of $V_2$ are all regular level sets.

Let $G(t)$ be the flow generated by $X$. By using $G(t)$, we can biholomorphically push $V_1$ and $V_2$ sufficiently close to the attraction point $p$. So without loss of generality, we may assume there exists some $t_0<0$ such that $G(-t_0)(V_2)$ is well-defined and the closure of $V_1$ is contained in $G(-t_0)V_2$.
Let $S_1$ be the boundary of $V_1$, $S_2$ be the boundary of $V_2$.
Consider the sets $G(-t)(S_2)$ for $0\leq t\leq t_0$. Then it is clear that they are all strictly pseudoconvex.
On $G(-t_0)B\backslash V_2$, we can find a smooth function $g_1$ satisfying
\begin{itemize}
\item $g_1 = 0$ on $S_2$.
\item $g_1$ is constant on the sets $G(-t)(S_2)$ for each $t$.
\item $g_1$ is strictly decreasing along the vector field $X$.
\end{itemize}
We can find an increasing convex function $u$ so that $u(g_1)$ is strictly psh.
Set $f_1 = u(g_1)$.
Similarly, we can construct a strictly psh function $f_2$ which satisfies
\begin{itemize}
\item $f_2$ is constant on sets $G(t)(S_1)$ for each $t\geq 0$. 
\item $f_2$ is strictly greater than the maximum of $f_1$ on $S_1$. 
\item $f_2$ is strictly less than $f_1$ on $S_2$.
\item $f_2$ is strictly decreasing along $X$.
\end{itemize}
More precisely, there exists $\delta>0$ so that $G(\delta)(V_1)$ contains the closure of $V_2$. Define a function $g_2$ which is constant on sets $G(t)(S_1)$ for each $t\geq 0$ and $g_2$ is strictly decreasing along $X$. Then we find an increasing convex function $v$ so that $v(g_2)$ is strictly psh. By subtracting a large number, we may assume $v(g_2)<-100$ for all $t>\delta$.
Now let $w$ be an increasing smooth convex function satisfying that $w(y) = y$ for $y\leq v(g_2(G(\delta)(S_1))$; $w$ is increasing sufficiently fast so that $w(v(g_2(G(\frac{\delta}{2})S_1)))>\sup f_1$ on $V_1$. Set $f_2 = w(v(g_2))$. Then $f_2$ satisfies the conditions above.
Now let $f$ be the max of $f_1$ and $f_2$. We can mollify the function $f$ so that it is strictly psh and has no critical point, since the derivatives of $f_1$, $f_2$ along $X$ are all strictly negative. For this part, check corollary $3.20$ of \cite{[E]}.
This concludes the proof of proposition \ref{prop11}.
\end{proof}

We can apply proposition \ref{prop11} to $A_i$ and $D$ which are introduced above lemma \ref{lm3}. As $D$ is contactomorphic to the standard sphere, $A_i$ and $A$ are all contactomorphic to the standard sphere. We need the following result of Mclean:
\begin{prop}\label{prop12}\cite{[Mc]}
Let $V$ be a normal variety of dimension $3$ with isolated singularity $p$. Assume the link of $V$ is contactomorphic to the standard sphere $\mathbb{S}^5$, then $p$ is a smooth point.
\end{prop}
 We just apply Mclean's theorem for the case $k=n-3$.
This concludes the proof of proposition \ref{prop10}.
\end{proof}
\begin{remark}
Corollary $1.4$ of \cite{[Mc]} states that if the so called Shokurov conjecture in algebraic geometry (we skip the statement) is true, then proposition \ref{prop12} holds for any dimension. As a consequence, if the Shokurov conjecture is true, $X$ is a complex manifold, i.e., no singularity appears. 
\end{remark}
Now we consider the setting as in theorem \ref{thm2}.
Pick a point $q$ in the limit space $X$. If a tangent cone splits off $\mathbb{R}^{2k}$ where $k\geq n-3$, then according to proposition \ref{prop10},
the tangent cone is complex analytically smooth. By lifting the holomorphic chart to $M_i$ as in claim \ref{cl9},  we find that $q$ is a complex analytically smooth point on $X$. 
According to the dimension estimate of singular set in Cheeger-Colding \cite{[CC2]}, metric singularities whose tangent cones do not split $\mathbb{R}^{2n-6}$ have Hausdorff dimension at most $2n-8$. Since the distance function induced by holomorphic coordinates are bounded by the metric (each coordinate function has locally bounded gradient), we find that the complex analytic singularity of $X$ has complex codimension at least $4$. This concludes the first part of theorem \ref{thm2}.

\medskip
For the second part of theorem \ref{thm2}, we use a similar argument as in Perelman's proof of the stability theorem \cite{[P]}\cite{[K]}.
\begin{prop}\label{prop13}Given any $n\in\mathbb{N}$, $v>0$, there exist $\epsilon=\epsilon(n, v)>0, \delta=\delta(n, v)>0$ so that the following holds.
Let $(M_i, p_i)$ be a sequence of pointed K\"ahler manifolds converging to $(X, p)$ in the Gromov-Hausdorff sense. Assume the bisectional curvature of $M_i$ is bounded from below by $-\epsilon^3$ and $vol(B(p_i, \frac{1}{\epsilon}))\geq \frac{v}{\epsilon^{2n}}$. Assume that $B(p, \frac{1}{\epsilon})$ is $\epsilon$-Gromov-Hausdorff close to $B_V(o, \frac{1}{\epsilon})$, where $(V, o)$ is a metric cone isometric to $\mathbb{R}^{2k}\times W$. Then there exists an open set $B(p, \delta)\subset U\subset B(p, 1)$ so that $U$ is biholomorphic to a product $\mathbb{D}^k\times Z$ where $Z$ is an irreducible normal complex analytic space.  
\end{prop}
\begin{proof}
The proof is almost the same as in proposition \ref{prop9}. We skip the argument here. \end{proof}
The following is a local stability result.
\begin{prop}\label{prop14}
Under the same assumptions of proposition \ref{prop13}, there exists $\gamma=\gamma(n, v)>0$ so that we can find a homeomorphism $\Phi_i$ from an open set $B(p_i, \gamma)\subset U_i\subset B(p_i, 1)$ to $B(p, \gamma)\subset U\subset B(p, 1)$ respecting the holomorphic factor $\mathbb{D}^{k}$. Also $\Phi_i$ is a $\Phi(\epsilon)$-Hausdorff approximation of subsets in $\mathbb{C}^N$, if we consider the embeddings as in lemma \ref{lm2} (the conical structure of $V$ in lemma \ref{lm2} is not essential). Thus $X$ is a topological manifold.
\end{prop}
\begin{proof}
We use reverse induction. If $k=n$, then the conclusion follows from proposition \ref{prop10}. Assume the proposition is proved for $k\geq j+1$. We need to prove it for $k\geq j$. Let $\epsilon_{j+1}$, $\delta_{j+1}$ be the constants in proposition \ref{prop14} and proposition \ref{prop13} corresponding to $k\geq j+1$.

Just assume $k=j$.  By Gromov compactness, we can find small $a=a(n, v)>0$ so that $B(p_i, a)$ and $B(p, a)$ are all embedded in $\mathbb{C}^N$. Let $H_i, z^i_s, z_s$ be defined as in proposition \ref{prop10} (recall $z^i_s, z_s$ are defined in the beginning. $H_i$ is defined right above claim \ref{cl8}). Since $X$ is not necessarily a metric cone, $z_s$ are not necessarily homogeneous. However, by a compactness argument, if $\epsilon$ is small, we may assume that they are almost homogeneous in the sense that $|Z(z_s)-\alpha_sz_s|<\rho$. Here $\rho$ is an arbitrarily prescribed small number, $Z$ is the limit of the holomorphic vector field $Z_i$.

 Let $Z^i_j$ be zeros of $z^i_1, ..., z^i_j$, $\Sigma$ be the zeros of $z_1, .., z_j$.
Let us assume $H_i$ converges to a holomorphic vector field $H$ on $\Sigma$.  
Let $E$ be the Euclidean distance function on $\mathbb{C}^N$. We may assume that $(z_1, .., z_N)_i^{-1}(B_E(0, \lambda))\subset\subset B(p_i, \frac{a}{2})$, where $\lambda$ depends only on $n, v$. 
Let $S_i=\partial B_E(p_i, \lambda)\cap Z^i_j$,  $S=\partial B_E(p, \lambda)\cap \Sigma$.  Now we fix the value $\lambda=\lambda(n, v)$.

Pick a point $q\in S$. Consider points $S_i\ni q_i\to q$. If $\epsilon$ is sufficiently small, we can find $\delta^1_j, \delta^2_j$ depending only on $n, v$ so that for some $\delta^2_j<\delta^0_j<\delta^1_j$,  $B(q, \frac{\delta^0_j}{\epsilon_{j+1}})$ and $B(q_i, \frac{\delta^0_j}{\epsilon_{j+1}})$ are $\epsilon_{j+1}\delta^0_j$-Gromov-Hausdorff close to a geodesic ball in a metric cone which splits off $\mathbb{R}^{2j+2}$. By proposition \ref{prop8} and proposition \ref{prop13}, we find some open set $B(q_i, \delta_{j+1}\delta^0_j)\subset U_i\subset B(q_i, \delta^0_j)$ $(B(q, \delta_{j+1}\delta^0_j)\subset U\subset B(q, \delta^0_j))$ so that $U_i$ $(U)$ is biholomorphic to $\mathbb{D}^{j+1}\times \hat{Z}_i$ $(\hat{Z})$. Say the coordinate on $\mathbb{D}^{j+1}$ is given by $(z^i_1, .., z^i_j, w^i)$ $((z_1, .., z_j, w))$. Furthermore, $z^i_1(q_i) = \cdot\cdot\cdot = z^i_j(q_i) = w^i(q_i) =0$ $(z_1(q) = \cdot\cdot\cdot = z_j(q) = w(q) =0)$, where $w^i$ $(w)$ come from the splitting along gradient of distance to $p_i$ $(p)$, roughly speaking. We may assume that $z^i_s\to z_s$, $w^i\to w$.

 By the induction hypothesis, there exists a homeomorphism $\Phi_i$ from $U_i$ to $U$ respecting the holomorphic factor $\mathbb{D}^{j+1}$. Write the coordinate function $w^i = x_i+\sqrt{-1}y_i$ $(w = x+\sqrt{-1}y)$. By a compactness argument, we find some $c=c(n, v)>0$ so that if $\epsilon$ is small enough depending only on $n, v$, 
 \begin{equation}\label{0}-Re H_i(d_E(\cdot, 0))<-c<0, -Re H_i(x_i)<-c<0\end{equation} on $S_i$ for all large $i$. Let $G_i$ be the subset of $U_i$ which is given by the common zeros of $z^i_1, .., z^i_j, x_i$. We can project $G_i$ and $\Phi_i(G_i)$ to $S_i$ and $S$ by the flow generated by $Re H_i$ and $Re H$. From (\ref{0}), we see that the projections are all local homeomorphisms. Therefore, we have a homeomorphism from local parts of $S_i$ to $S$. This implies that $S$ is a topological manifold. Moreover, by simple ode argument, one can verify that this is a Hausdorff approximation. Furthermore, if we have two points which are close to each other on $S_i$, then we can connect them by the shortest geodesic on $M_i$ with small length. Therefore, the diameter is small. We project the curve to $S_i$ by using the flow generated by $H_i$. By ode argument, we see that the projected curve still has small diameter. This implies that there exists a function $\xi$ as in definition \ref{def2} so that $S_i$ are all $\xi$-connected for all sufficiently large $i$. Thus $S$ is also $\xi$-connected.

By applying the gluing theorem (proposition \ref{prop3}), we have a homeomorphism from $S_i$ to $S$ which is also a Hausdorff approximation. By using the flow generated by $-Re H_i$ and $-Re H$, we can extend this as a homeomorphism for domains bounded by $S_i$ and $S$ on $Z^i_j$ and $\Sigma$ . This is still a Hausdorff approximation. The product with $\mathbb{D}^j$ becomes a homeomorphism. Recall holomorphic vector field $W_m$ in claim \ref{cl6} is used to construct the biholomorphism in proposition \ref{prop8}. Since these holomorphic vector fields have a convergence subsequence, so this is a Hausdorff approximation. This completes the induction.
\end{proof}

Applying proposition \ref{prop3} again, we can glue local homeomorphisms to a global homeomorphism if $X$ is compact. The proof of theorem \ref{thm2} is complete.
\medskip
\end{proof}

\section{\bf{Proof of the main theorem}}
 Let $\mathcal{O}_d(M)$ denote polynomial growth holomorphic functions on $M$ with degree bounded by $d$. Let $\mathcal{O}_P(M) = \cup_{d>0}\mathcal{O}_d(M)$.
By choosing a large $D$, we may assume $\mathcal{O}_D(M)$ embeds $M$ to $\mathbb{C}^{N-1}$. Here $N = dim(\mathcal{O}_D(M))$. That is, we ignore the constant function in the holomorphic embedding.
 Consider a blow down sequence $(M_i, p_i, d_i) = (M, p, \frac{d}{r_i})\to(M_\infty, p_\infty, d_\infty)$ in the Gromov-Hausdorff sense. Here $r_i$ is a sequence increasing to infinity. 
\begin{prop}\label{prop15}
For any $d>0$, $dim(\mathcal{O}_d(M))=dim(\mathcal{O}_d(M_\infty))$. \end{prop}
\begin{proof}
The proof is in fact contained in \cite{[L3]}\cite{[L2]}.  We only give a sketch. First we prove $dim(\mathcal{O}_d(M))\leq dim(\mathcal{O}_d(M_\infty))$. Define an inner product $\langle f, g\rangle = \dashint_{B(p_i, 1)}f\overline{g}$ on $\mathcal{O}_d(M)$. Apply the three circle theorem and pass to limit for these functions. This concludes the proof of the first inequality. For details, see lemma $2$ of \cite{[L2]}.

For the reverse inequality, define a norm on $\mathcal{O}_d(M_\infty)$ by $\langle f, g\rangle = \dashint_{B(p_\infty, 1)}f\overline{g}$. Also define a norm on $\mathcal{O}_d(M)$ by $\langle u, v\rangle = \dashint_{B(p, 1)}u\overline{v}$. Let $f_1, .., f_s$ be a basis of $\mathcal{O}_d(M_\infty)$.
For sufficiently large $i$, we may lift $f_j(1\leq j\leq s)$ to $B(p_i, 1)$, say $f^i_j$. It is clear that $f^i_j$ are linearly independent.
We can find constants $c_{ijk}$ so that $F_{ik} = \sum\limits_jc_{ijk}f^i_j$ satisfies $\int_{B(p, 1)}F_{ik}\overline{F_{il}} = \delta_{kl}$. 
We look at the quotient \begin{equation}\label{eq45}\frac{\sup\limits_{B(p, \frac{r_i}{2})}|F_{ik}|}{\sup\limits_{B(p, \frac{r_i}{3})}|F_{ik}|}=\frac{\sup\limits_{B(p_i, \frac{1}{2})}|F_{ik}|}{\sup\limits_{B(p_i, \frac{1}{3})}|F_{ik}|}=\frac{\sup\limits_{B(p, \frac{r_i}{2})}|\sum\limits_jc_{ijk}f^i_j|}{\sup\limits_{B(p, \frac{r_i}{3})}|\sum\limits_jc_{ijk}f^i_j|}.\end{equation} By dividing by the supremum of $c_{ijk}$ (fix $i, k$), we may assume that the maximal coefficient in the last part of (\ref{eq45}) is equal to $1$. As $f^i_j$ are linearly independent, by a compactness argument, we find that for sufficiently large $i$, (\ref{eq45}) is bounded by $d+\epsilon$ for any $\epsilon>0$. Let $i\to\infty$ and apply the three circle theorem, the functions $F_{ik}$ converge to linearly independent functions on $\mathcal{O}_d(M)$. 

\end{proof}

Let us apply some argument in \cite{[DS2]}.
By dimension estimate for $\mathcal{O}_d(M)$, we can find a strictly increasing sequence $d_1, d_2, d_3, ...$ so that for any $d$ satisfying $d_s\leq d<d_{s+1}$, $\mathcal{O}_{d_s}(M)=\mathcal{O}_{d}(M)\neq \mathcal{O}_{d_{s+1}}(M)$. Let us choose $f_{s, l}\in \mathcal{O}_{d_s}(M) (l = 1, 2, .., l_s)$ so that they form a basis of $\mathcal{O}_{d_{s}}(M)/\mathcal{O}_{d_{s-1}}(M)$ as quotient of vector spaces. Set $W_s = $span$\{f_{s, 1}, .., f_{s, l_s}\}$. Let $f^i_{s, l}$ be an orthonormal basis of $W_s$, with respect to the $L^2$ integration on $B(p_i, 1)$. After taking subsequences, we may assume $f^i_{s, l}\to f^\infty_{s, l}$ uniformly on each compact set of $M_\infty$.
\begin{claim}\label{cl-082}
 $f^\infty_{s, l}$ is homogeneous of degree $d_s$. 
 \end{claim}
 \begin{proof}
First, $f^\infty_{s, l}\in\mathcal{O}_{d_s}(M_\infty)$ by three circle theorem. Second, for any $\epsilon>0$, there exists $R>0$ depending on $\epsilon$ so that for any function $u\in W_s$, $N(R, u) = \frac{\sup\limits_{B(p, 2R)}|u|}{\sup\limits_{B(p, R)}|u|}\geq 2^{d_s} - \epsilon$. To prove this, write $u = c\sum a_lf_{s, l}$ where $\sup |a_l| = 1$. We may assume $c=1$ by scaling. By definition of $W_s$, for each $u$, we can find $R_u$ so that $N(R_u, u)\geq 2^{d_s}-\frac{\epsilon}{2}$. By continuity, if $v = \sum b_lf_{s, l}$ and $|a_l-b_l|$ is sufficiently small, $N(R_u, v)\geq 2^{d_s}-\epsilon$. Three circle theorem implies that $N(r, u)$ monotonic increasing. Now the second point follows from the compactness of $\mathbb{CP}^{l_s-1}$ (we are thinking the coefficients lives in $\mathbb{CP}^{l_s-1}$). We can apply the argument of proposition $5$ in \cite{[L5]} to finish the proof of the claim.

\end{proof}

Let $Z$ be the vector space of holomorphic vector fields $X$ on $M$ so that $X(\mathcal{O}_d(M))\subset \mathcal{O}_d(M)$ for all $d$. This means for any $f\in\mathcal{O}_d(M)$, the derivative $X(f)\in\mathcal{O}_d(M)$. Finite generation of $\mathcal{O}_P(M)$ and linear algebra imply $Z$ has finite dimension. Similarly, let $Z_\infty$ be the vector space of holomorphic vector fields $Y$ on $M_\infty$ so that $Y(\mathcal{O}_d(M_\infty))\subset \mathcal{O}_d(M_\infty)$ for all $d$.

Let $f_1, .., f_N$ be a basis for $\mathcal{O}_D(M)$. Assume $f^i_1, .., f^i_N$ be a new basis so they are orthonormal with respect to the $L^2$ integration on $B(p_i, 1)$.
Let $X_1, ..., X_k$ be a basis of $Z$. We can find new basis $X^i_1, ..., X^i_k$ so that they are orthnormal with respect to the Hermitian inner product defined by $\langle X_a, \overline X_b\rangle_i = \dashint_{B(p_i, 1)}\sum\limits_{j=1}^N \langle X_a(f^i_j), \overline{X_b(f^i_j)}\rangle$. We can similarly define a Hermitian inner product on $Z_\infty$ by $\langle Y_a, \overline Y_b\rangle_\infty = \dashint_{B(p_\infty, 1)}\sum\limits_{j=1}^N \langle Y_a(f^\infty_j), \overline{Y_b(f^\infty_j)}\rangle$. Here $f^\infty_j$ is the limit of $f^i_j$.

\begin{definition}For any fixed $R>0$, 
We say vector fields $X_i$ on $B(p_i, 2R)$ converges to $X_\infty$ on $B(p_\infty, 2R)$, if for and $d>0$, any $f_i\in\mathcal{O}_d(M_i)$ with $f_i\to f_\infty$ on $B(p_\infty, 2R)$, $X_i(f_i)\to X_\infty(f_\infty)$ uniformly on $B(p_\infty, R)$. 
\end{definition}

It is clear from the definition that the Hermitian inner product $\langle \cdot, \cdot \rangle_i$ on $Z$ converges to the inner product $\langle \cdot, \cdot \rangle_\infty$ on $Z_\infty$.
Note by three circle theorem, after taking subsequence, $X^i_1, .., X^i_k$ converge to $X^\infty_1, X^\infty_2, ..., X^\infty_k$ on $M_\infty$. Moreover, $X^\infty_j(\mathcal{O}_d(M_\infty))\subset \mathcal{O}_d(M_\infty)$ for all $d$. Note corollary \ref{cor00} states $-Re Z_i\to -r\frac{\partial}{\partial r}$ (recall $Z_i$ was defined in (\ref{eq0})). 

The following claim is crucial. The argument is in the same spirit as claim $6.1$ of \cite{[L3]}.
\begin{claim}\label{cl-45}
The complexification of $-r\frac{\partial}{\partial r}$ is in the span of $X^\infty_1, ..., X^\infty_k$.
\end{claim}
 \begin{proof}
It is clear that the complexification of $-r\frac{\partial}{\partial r} \in Z_\infty$.
Assume the claim is not true. After orthogonalization via $\langle \cdot, \cdot \rangle_i$, we can find a basis $X^i_1, .., X^i_k, X^i_{k+1}$ of span of $X_1, .., X_k, Z_i$ so that $X^i_j\to X^\infty_j$ for all $j= 1, .., k+1$.
Here we may require $X^i_j (1\leq j\leq k)$ be the same as defined above the claim. We further require that $X^\infty_1, .., X^\infty_{k+1}$ be linearly independent.

We can also just diagonalize the span of $X_1, ..., X_k, Z_i$ on $B(p, 1)$. This is just given by the $L^2$ integration on $B(p, 1)$. Say the new basis is given by $Z_1^i, .., Z^i_{k+1}$ and $Z_j^i = \sum\limits_{s=1}^{k+1} a_{ijs}X^i_s$.
We assert that by taking subsequence, $Z^i_1, ..., Z^i_{k+1}$ converge uniformly, as holomorphic vectors on each compact set of $M$, to holomorphic vector fields in $Z$. If this is proved, we have a contradiction with that $dim(Z) = k$.

To prove the assertion,
let $f\in\mathcal{O}_d(M)$. Let $s_j^i = \frac{\max\limits_{B(p_i, 2)}|Z_j^i(f)|}{\max\limits_{B(p_i, 1)}|Z_j^i(f)|}=\frac{\max\limits_{B(p_i, 2)}|(\sum a_{ijs}X^i_s)(f_i)|}{\max\limits_{B(p_i, 1)}|(\sum a_{ijs}X^i_s)(f_i)|}$. Here $f_i = c_if$ where $c_i$ is a constant so that the $L^2$ norm of $f_i$ is $1$ on $B(p_i, 1)$. Assume $f_i\to f_\infty$.
We can also find a constant $c_i'$ so that $b_{ijs} = c_i'a_{ijs}$ satisfy $\max\limits_{s} |b_{ijs}| = 1$. Say $b_{ijs}\to b_{\infty js}$.

\medskip

Case 1:

\medskip

If $\sum b_{\infty js}X^\infty_s(f_\infty)\neq 0$, $s^i_j\leq 2^d+\epsilon$ for all sufficiently large $i$. Three circle theorem implies that $Z_j^i(f)$ converges to a function in $\mathcal{O}_d(M)$.

\medskip

Case 2:

\medskip

If $\sum b_{\infty js}X^\infty_s(f_\infty)= 0$, then for some large $e$, we can just find $g_i\in W_e$ so that $\int_{B(p, 1)} |g_i|^2 = 1$ and after normalization on $B(p_i, 1)$ ($h_i$ is constant so that $\hat{g}_i = h_ig_i$ satisfies $\int_{B(p_i, 1)} |\hat{g}_i|^2 = 1$), $\hat{g}_i\to g_\infty$ on $B(p_\infty, 1)$ and $\sum b_{\infty js}X^\infty_s(g_\infty)\neq 0$. By using the argument in case $1$,
after passing to subsequence, we may assume $g_i\to g\in W_e, Z^i_j(g_i)\to u_j\in\mathcal{O}_{d_e}(M), Z^i_j(g_if)\to v_j\in\mathcal{O}_{d+d_e}(M)$.
Now $Z^i_j(f) = \frac{Z_j^i(fg_i) - Z^i_j(g_i)f}{g_i} \to \frac{v_j-fu_j}{g}$. Note the convergence is uniform on each compact set. There is no problem near the zero of $g_i$ or $g$ (just apply the Cauchy estimate).
\begin{lemma}
$\mu_j=\frac{v_j-fu_j}{g}\in\mathcal{O}_d(M)$.
\end{lemma}
\begin{proof}
Observe the numerator has order $d+d_e$. We need the following result in \cite{[Mo]}.
\begin{prop}[Mok]
Let $f, g$ be polynomial growth holomorphic functions on a complete K\"ahler manifold $M$ with $Ric\geq 0$. Suppose $h=\frac{f}{g}$ is holomorphic, then $h$ is of polynomial growth.
\end{prop}
\begin{proof}
Let us say $f(p), g(p)\neq 0$.
Set $F_1(x) = \log |f(x)|^2+\int_{B(p, R)}G_R(x, y)\Delta\log |f(y)|^2, F_2(x) = \log |g(x)|^2+\int_{B(p, R)}G_R(x, y)\Delta\log |g(y)|^2$.
\begin{lemma}
For large $R$ and $i = 1, 2$, on $B(p, \frac{R}{2})$, $-C\log R\leq F_i(x)\leq C\log R$.
\end{lemma}
\begin{proof}
It is clear that $F_i(x)$ is harmonic on $B(p, R)$. Now maximum principle says that $F_i(x)\leq C\log R$ on $B(p, R)$. Let $H_i  = C\log R-F_i\geq 0$. Then gradient estimate implies that on $B(p, \frac{3}{4}R)$, $|\nabla \log H_i|\leq \frac{C_1}{R}$. Observe $H_i(p)\leq C\log R$.  Then the harnack inequality implies that $H_i\leq C_2\log R$ on $B(p, \frac{R}{2})$. This completes the proof of the claim.
\end{proof}
It is clear that on $B(p, \frac{R}{2})$, $\log |h(x)|^2 \leq F_1(x)-F_2(x)\leq C\log R$ ($C$ is independent of $R$). The proof of the proposition is complete.\end{proof} 

We come back to the proof of the lemma.
Let us assume $a$ is the smallest number so that $\mu_j\in\mathcal{O}_a(M)$. Then three circle theorem says $\lim\limits_{r\to\infty}\frac{M(\mu_j, 2r)}{M(\mu_j, r)} = 2^a$ where $M(\mu_j, r) = \sup\limits_{B(p, r)}|\mu_j|$.
 Assume the lemma is not true. Then $a>d$. As $g\in W_e$, we can apply claim \ref{cl-082}. After normalizing the functions $\mu_j$ and $g$ on $B(p_i, 1)$ and taking limits, we find their product of the limits, converges to a  homogeneous function of degree $a+d_e$ on $M_\infty$. However, this contradicts that $v_j-fu_j\in\mathcal{O}_{d+d_e}$.

We conclude that $Z^i_j(f)$ converges to a holomorphic function of degree $d$, for any $f\in\mathcal{O}_d(M)$. This implies that $Z^i_j$ converges to an element in $Z$. The assertion is proved. 
\end{proof}

The proof of claim \ref{cl-45} is complete.
\end{proof}

Recall $\mathcal{O}_D(M)$ embeds $M$ to $\mathbb{C}^{N-1}$. 
By applying claim \ref{cl-082},  we can find basis $f^j_i$ of $\mathcal{O}_D(M)$ so that they are almost orthonormal on $B(p_i, 1)$ and the limits are all homogeneous. Let us take for granted that $f^N_i = 1$ (this is the constant function in $\mathcal{O}_D(M)$. Given claim \ref{cl-45}, we can find $X_i\in Z$ so that $X_i$ converges to the complexification of $-r\frac{\partial}{\partial r}$ on any compact set of $M_\infty$.
In particular, $X_i(f^j_i)\to 2(-r\frac{\partial }{\partial r}) f^j_\infty = -2d(f^j_\infty)f^j_\infty$. In the last equality, we have used that $f^j_\infty$ are homogeneous. By using the basis $f^j_i$ of $\mathcal{O}_D(M)$, we find the action of $X_i$ on $\mathcal{O}_D(M)$ is given by \begin{equation}\label{eq-34}X_i(\vec{P_i}) = A_i\vec{P_i} + \vec{C_i},\end{equation} where $\vec{P_i}$ is the column vector $(f^1_{i}, .., f^{N-1}_i)^T$, $\vec{C_i}$ is a constant $(N-1)\times 1$ vector.
$A_i$ is a constant $(N-1)\times (N-1)$ matrix (depending on $i$) which satisfies that all real parts of diagonal elements are less than or equal to $-\frac{1}{2}$ (recall that the degree of a nonzero homogeneous holomorphic function on $M_\infty$ is at least $1$), also the off diagonal elements are very small. In particular, all real part of eigenvalues of the matrix are strictly negative.
Let us fix a sufficiently large $i_0$. We find the corresponding $X=X_{i_0}\in Z$.
\begin{claim}
$X$ is an integrable vector field on $M$. Moreover, $X$ retracts $M$ to a point.
\end{claim}
\begin{proof}
This is just linear algebra. Indeed, the action of $X$ on $M$ could be seen from $(\ref{eq-34})$.
We extend the vector field in the natural way to $\mathbb{C}^{N-1}$ which we still call $X$.
Since all real part of eigenvalues are strictly negative, the flow $\sigma_t$ generated by $X$ must retract $\mathbb{C}^{N-1}$ to a point, say $o$. Then $o\in M$.

\end{proof}

To complete the proof of the theorem, we can just apply the Poincare-Dulac normal coordinate \cite{[JKS]}.
The result (page $1190$ of \cite{[JKS]}) says that we can find a local holomorphic chart $U = U(z_1, ..., z_n)$ near $o$(the unique fixed point) so that $U$ is the unit ball in $\mathbb{C}^n$ (measured in Euclidean coordinate $(z_1, ..., z_n)$) and $X = -\sum\limits_{j=1}^n (\lambda_jz_j+g_j(z))\frac{\partial}{\partial z_j}$,
where 
\begin{itemize}
\item $0<Re \lambda_1 \leq Re \lambda_2 \leq \cdot\cdot\cdot\leq Re\lambda_n$
\item $g_1\equiv 0$
\item For every $j\in\{2, ..., n\}$, $g_j(z)$ is a polynomial of $z_1, .., z_{j-1}$ only, vanishing at the origin. If the identity $\lambda_j = \sum\limits_{k=1}^{j-1}m_k\lambda_k$ holds for some nonnegative integers $m_k$, then the condition $g_j(e^{\lambda_1t}z_1, ..., e^{\lambda_{j-1}t}z_{j-1}) = e^{\lambda_jt}g_j(z_1, ..., z_{j-1})$. If $\lambda_j = \sum\limits_{k=1}^{j-1}m_k\lambda_k$ never hold for nonnegative integers $m_k$, $g_j = 0$.

\end{itemize}

On $M'=\mathbb{C}^n$, we can define a holomorphic vector $\hat{X} = -\sum\limits_{j=1}^n (\lambda_jz_j+g_j(z))\frac{\partial}{\partial z_j}$, where $g_j$ is the same polynomial as in $X$. By ode, one can prove that $\hat{X}$ is integrable. 
Let $\hat\sigma_t(z)$ be the flow generated by $\hat{X}$ on $\mathbb{C}^n$. Then one can verify that $\hat\sigma_t$ is a retracting holomorphic vector field on $M'$ with the origin as the unique fixed point.

$U$ is an open set of $M$. Let us identify it with the unit ball in $M'=\mathbb{C}^n$. 
Define a map $F: M'=\mathbb{C}^n\to M$ as follows:
Given any $z\in M'=\mathbb{C}^n$, we can find sufficiently large $t$ so that $\hat\sigma_t(z)$ is contained in the unit ball. Then define $F(z) = \lim\limits_{t\to+\infty}\sigma_{-t}(\hat\sigma_t(z))$. Since the vector field $\hat{X}$ on the unit ball of $M'$ is the same as $X$ in $\mathbb{C}^n$, we obtain that $F(z)$ is well defined (independent of the value of $t$, as $t$ is sufficiently large). It is clear that $F$ is holomorphic and invertible: $F^{-1}(y) = \lim\limits_{t\to+\infty}\hat\sigma_{-t}(\sigma_t(y))$. Thus $M$ is biholomorphic to $\mathbb{C}^n$.

Now let us check that these coordinate functions $z_1, ..., z_n$ are of polynomial growth on $M$. 
 We use induction on the degree of $Re\lambda_s$.
Assume $z_s$ are all of polynomial growth for $Re\lambda_s\leq h$.  Let $h_1>h$ be so that there exists some $j$ with $Re\lambda_j = h_1$ while there is no $Re\lambda_j$ between $h$ and $h_1$.
Assume for $j = j_1, .., j_1+k-1$, $z_j$ satisfy $Re\lambda_j = h_1$. 

Let $\mathcal{O}'_D(M)$ be the subset of $\mathcal{O}_D(M)$ which vanish at $o$ (recall this is the unique fixed point of the flow generated by $X$). Then $X(\mathcal{O}'_D(M)) \subset \mathcal{O}'_D(M)$. Let $f_1, .., f_{N-1}$ be the basis of the Jordan form for the action of $X$ on $\mathcal{O}'_D(M)$. We claim that each $f_s$ is a polynomial of $z_1, .., z_n$. Given a monomial $z_1^{i_1}\cdot\cdot\cdot z_n^{i_n}$, define the weight $w$ as $\lambda_1i_1+\cdot\cdot\cdot+\lambda_ni_n$. Since $Re\lambda_s>0$, given any $c\in\mathbb{R}$, there are at most finitely many monomials (up to a factor) so that the real part of $w$ is no greater than $c$. Note the action of $X$ on monomials preserves the weight. Let $V_w$ be the span of monomials with weight $w$. Then each $V_w$ is finite dimensional.

Assume $f_s$ (generalized eigenvector) corresponds to eigenvalue $\lambda$. By Taylor expansion at $o$ and Cayley-Hamilton theorem, 
we see $f_s\in V_\lambda$. In particular, $f_s$ is a polynomial of $z_1, .., z_n$. Since $f_1, .., f_{N-1}$ gives the embedding of $M$ to $\mathbb{C}^{N-1}$, we can always find $f_{l_1}, ..., f_{l_k}$ so that $\det(\frac{\partial f_{l_s}}{\partial z_{j}})|^{j = j_1, ..., j_1+k-1}_{s=1, .., k}\neq 0$ at $0$. In particular, these $f_{l_s}$ must satisfy that the real part of the eigenvalue is equal to $h_1$. According to induction, there exists an invertible $k\times k$ matrix $A$ so that 
$f_{l_s} =\sum_j A_{sj}z_j+B_s$, where each $B_s$ has polynomial growth. Thus $z_j$ has polynomial growth. The induction is completed.

As any function in $\mathcal{O}_D(M)$ is a polynomial of $z_1, .., z_n$, we see that $\mathcal{O}_P(M)$ is generated by $n$ polynomial growth holomorphic functions $z_1, .., z_n$. We can say in this way, $M$ is isomorphic to $\mathbb{C}^n$.
Thus the main theorem is proved.

\end{proof}

\end{document}